\documentclass[11pt]{article}

\usepackage[margin=3.5cm]{geometry}
\usepackage{amsmath,amsfonts,amssymb,amsthm}
\usepackage{tikz}
\usepackage{enumerate}
\usepackage{hyperref}
\usepackage{caption}
\usepackage{booktabs} % for tables
\usepackage{marvosym}

\newcommand{\xleftrightarrows}[2]{%
  \mathrel{\substack{
    \xrightarrow{#1} \\[-0.9ex] % Abstand hier feinjustieren
    \xleftarrow[#2]{}
  }}
}

\theoremstyle{plain}
\newtheorem{theorem}{Theorem}
\newtheorem{proposition}[theorem]{Proposition}
\newtheorem{lemma}[theorem]{Lemma}
\newtheorem{corollary}[theorem]{Corollary}
\newtheorem{fact}[theorem]{Fact}

\theoremstyle{definition}
\newtheorem{definition}[theorem]{Definition}
\newtheorem*{remark}{Remark}
\newtheorem{example}{Example}
\newtheorem*{running}{Example}

\newcommand\R{\mathbb{R}}

\DeclareMathOperator{\im}{im} % like \ker
\DeclareMathOperator{\sign}{sign}

\DeclareMathOperator{\e}{e} % like \ln
\DeclareMathOperator{\diag}{diag}

\DeclareMathOperator{\id}{Id}
\newcommand{\ones}{1}
\newcommand{\trans}{\mathsf{T}}
\newcommand{\la}{\lambda}

\DeclareMathOperator{\supp}{supp}
\DeclareMathOperator{\aff}{aff}

\newcommand{\dd}[2]{\frac{\text{d} #1}{\text{d} #2}}

\definecolor{darkgreen}{rgb}{0.0,0.7,0.0}
\newcommand\blue[1]{{\color{black}#1}}
\newcommand\gray[1]{{\color{gray}#1}}
\newcommand\marginnote[1]{}

\newcommand\blfootnote[1]{%
  \begin{NoHyper}
  \renewcommand\thefootnote{}\footnote{#1}%
  \addtocounter{footnote}{-1}%
  \end{NoHyper}
}

\begin{document}

\title{Decomposable and essentially univariate mass-action systems:
Extensions of the deficiency one theorem}

\author{Abhishek Deshpande, Stefan M\"uller\textsuperscript{\Letter}}

\maketitle

\begin{abstract}
The classical and extended deficiency one theorems by Feinberg apply to reaction networks with mass-action kinetics that have independent linkage classes or subnetworks, each with a deficiency of at most one and exactly one \blue{absorbing} strong component. The theorems assume the existence of a positive equilibrium
and guarantee the existence of a unique positive equilibrium in every stoichiometric compatibility class.

In our work, we use the {\em monomial dependency} which extends the concept of deficiency.
First, we provide a dependency one theorem for parametrized systems of polynomial equations that are essentially univariate and decomposable. 
As our main result, we present a corresponding theorem for mass-action systems, which permits subnetworks with arbitrary deficiency and arbitrary number of \blue{absorbing} strong components. Finally, to complete the picture, we derive the extended deficiency one theorem as a special case of our more general dependency one theorem.

\vspace{2ex}

\vspace{2ex}
\end{abstract}

\blfootnote{
\scriptsize

\noindent
{\bf Abhishek Deshpande} \\
Center for Computational Natural Sciences and Bioinformatics, 
International Institute of Information Technology Hyderabad,
Hyderabad, Telangana 500032,
India \\[1ex]
{\bf Stefan~M\"uller} \\
Faculty of Mathematics, University of Vienna, Oskar-Morgenstern-Platz 1, 1090 Wien, Austria \\
\Letter \, st.mueller@univie.ac.at
}

% ========= ========= ========= ========= ========= ========= ========= =========

%\clearpage

\section{Introduction}

Many systems in chemistry and biology (particularly in ecology and epidemiology) and also in economics and engineering are modeled as polynomial or power-law dynamical systems. These models can be formulated as {\em reaction networks} with mass-action or generalized {\em mass-action kinetics}. They capture complex dynamic behaviors, such as multistationarity, oscillations, and chaos, and the corresponding bifurcations.

The most prominent results of (chemical) reaction network theory, as founded in the 1970's by Horn, Jackson, and Feinberg, are the {\em deficiency zero} and {\em one theorems} (for mass-action systems)~\cite{HornJackson1972,Horn1972,Feinberg1972,Feinberg1987,Feinberg1995b}.
In essence, the concept of deficiency captures affine dependencies between {\em complexes} (representing the left- and right-hand sides of the reactions).
The deficiency zero theorem assumes {\em weak reversibility} (and deficiency zero)
and guarantees the existence of a unique positive equilibrium in every {\em stoichiometric compatibility class} (invariant subspace) and for all {\em rate constants} (system parameters);
moreover, it ensures the asymptotic stability of this equilibrium.
%On the other hand,
The deficiency one theorem has several assumptions.
It applies to mass-action systems with {\em independent} linkage classes or subnetworks,
{\em each} with a deficiency of at most one and exactly one \blue{absorbing} strong component.
Additionally, the theorem assumes the existence of a positive equilibrium (for given rate constants)
and guarantees the existence of a unique positive equilibrium in every compatibility class
(but does not address stability).
Finally, weak reversibility ensures the existence of a positive equilibrium.

Both, the deficiency zero and one theorems,
have been difficult to improve upon.
Only after 2010,
the deficiency zero theorem has been extended,
namely from mass-action to {\em generalized} mass-action kinetics~\cite{MuellerRegensburger2012,MuellerRegensburger2014,Mueller2016,MuellerHofbauerRegensburger2019,CraciunMuellerPanteaYu2019,BorosMuellerRegensburger2020,MuellerRegensburger2023}.
In this work,
we will extend the deficiency one theorem, even in the setting of mass-action kinetics.

The deficiency one theorem was first fully stated in 1987~\cite{Feinberg1987},
%(see also the 1980 version~\cite{Feinberg1980}), 
but was not proved until 1995~\cite{Feinberg1995a}.
Its extension from independent linkage classes to independent subnetworks
was mentioned in~\cite{Feinberg1987}, and a proof was outlined in~\cite{Feinberg1995a}. 
Since then, the upper bound of one for both the deficiency and the number of \blue{absorbing} strong components (per independent linkage class or subnetwork) has not been addressed.
Only the assumption of existence (of a positive equilibrium)
has been investigated further~\cite{boros2010notes_1,boros2012notes_2,boros2013dependence,boros2013existence}. 
In \cite{boros2010notes_1,boros2012notes_2}, Boros provides an equivalent condition for the existence of a positive equilibrium for reaction networks that satisfy the assumptions of the deficiency one theorem. In~\cite{boros2013dependence}, he characterizes single linkage-class, deficiency one mass-action systems for which a positive equilibrium exists for all or some rate constants. 
Finally, in~\cite{boros2013existence}, Boros proves the existence of a positive equilibrium within every stoichiometric compatibility class for weakly reversible, deficiency one mass-action systems.
Using different methods,
this result can be further extended to weakly reversible (not necessarily deficiency one) mass-action systems~\cite{Boros2019}.

In this work, 
we extend the validity of the deficiency one theorem to a much broader class of mass-action systems. 
First, we consider independent subnetworks (instead of independent linkage classes) from the outset. 
Second, we do not assume that the subnetworks have a deficiency of at most one. 
Instead, we use the {\em monomial dependency} which generalizes the concept of deficiency. 
Specifically, dependency accounts for the fact that non-source vertices in a reaction network do not contribute monomials to the polynomial equations for the equilibria. 
As a consequence, the dependency is often smaller than the deficiency.
Third,
we do not assume that the subnetworks have exactly one \blue{absorbing strong component}.
%Instead, we bound only the number of \blue{absorbing strong components} that are not singletons.

We illustrate our results in a series of examples in Section~\ref{sec:examples}.
Here, we consider Example~\ref{exa2},
given by the reaction network
\[
0 \leftarrow X_1 \rightleftarrows X_1+X_2 \to X_2 \rightleftarrows 3\, X_1 
\]
or, equivalently, by the ``embedded'' graph
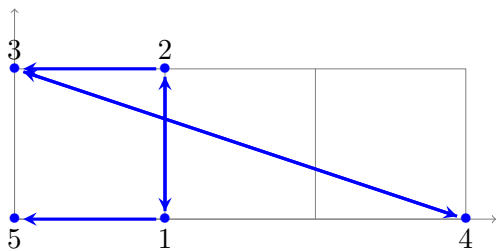
\begin{center}
\begin{tikzpicture}[scale=2]
    
\draw [step=1, gray, very thin] (0,0) grid (3,1);
\draw [ ->, gray, very thin] (0,0)--(3.25,0);
\draw [ ->, gray, very thin] (0,0)--(0,1.25);

\begin{scope}[shift={(0,0)}]

\node[circle, fill=blue, inner sep=1pt, outer sep=5pt] (A) at (1,0) {};
\node[circle, fill=blue, inner sep=1pt, outer sep=5pt] (B) at (1,1) {};
\node[circle, fill=blue, inner sep=1pt, outer sep=5pt] (C) at (0,1) {};
\node[circle, fill=blue, inner sep=1pt, outer sep=5pt] (D) at (3,0) {};
\node[circle, fill=blue, inner sep=1pt, outer sep=5pt] (E) at (0,0) {};

%\node[inner sep=0,outer sep=1] (0) at (0,0) {};

\draw[arrows={-stealth},very thick,blue,transform canvas={xshift=-2pt}] (A) to (B);
\draw[arrows={-stealth},very thick,blue,transform canvas={xshift=+2pt}] (B) to (A);
\draw[arrows={-stealth},very thick,blue] (B) to (C);
\draw[arrows={-stealth},very thick,blue] (C) to (D);
\draw[arrows={-stealth},very thick,blue,transform canvas={yshift=-3.5pt,xshift=-0.5pt}] (D) to (C);
\draw[arrows={-stealth},very thick,blue] (A) to (E);

% \node [below right] (A) at (A) {$1$};
% \node [above right] (B) at (B) {$2$};
% \node [above left]  (C) at (C) {$3$};
% \node [below right] (D) at (D) {$4$};
% \node [below left]  (E) at (E) {$5$};

%\node [below] (0) at (0) {$0$};

\end{scope}

\end{tikzpicture}
\end{center}
and assume mass-action kinetics.
The network has one independent linkage class/sub\-network.
Hence, its deficiency is
$\delta = |V| - 1 - \dim(S)=5-1-2 = 2$, 
where $V$ is the set of vertices and $S$ is the stoichiometric subspace (the linear span of the differences of complexes). 
However, its dependency is $d = |V_s| - 1 - \dim(L) = 4-1-2 = 1$, where $V_s$ is the set of source vertices
and $L$ is the monomial difference subspace (the linear span of the differences of source complexes).  
Moreover, this network has two \blue{absorbing} strong components, one singleton (complex~0) and one non-singleton (with complexes~$X_2$ and~$3\, X_1$). 
Clearly, it does not satisfy the conditions of the deficiency one theorem.
Still, our dependency one \blue{results} for mass-action systems,
Theorem~\ref{thm:main:mass-action} \blue{and Corollary~\ref{cor:main:mass-action},} can be applied.

Technically,
we treat positive equilibria of mass-action systems 
as {\em parametrized} systems of {\em polynomial equations} with {\em classes}
and apply recent fundamental results for such systems~\cite{MuellerRegensburger2023a}.
Specifically, their solution set is essentially the solution set on the {\em coefficient polytope} 
(modulo an exponential fiber involving the {\em monomial difference subspace}~$L$).
\blue{We assume (i) the reaction network can be decomposed into independent subnetworks, (ii) the equation system is {\em decomposable} in the terminology of~\cite{MuellerRegensburger2023a}, (iii) the subsystems (classes) have {\em monomial dependency} one and hence are univariate,
and (iv) $L=S$ (or $L=K$, where $K$ is the kinetic subspace).
Then, there exists a unique positive equilibrium in every stoichiometric (or kinetic) compatibility class.
Note that, even if $K=S$, the subnetworks need not have exactly one absorbing strong component.}

We first provide
a {\em dependency one theorem} for one class, Theorem~\ref{thm:one_class},
using the fundamental result from~\cite{MuellerRegensburger2023a},
Theorem~\ref{thm:previous}.
Next, we extend it to {\em decomposable systems} (with several classes) and obtain Theorem~\ref{thm:main}.
By applying the latter to reaction networks
(and using Birch's theorem),
we arrive at a dependency one theorem for mass-action systems, 
Theorem~\ref{thm:main:mass-action}.
Finally, this allows us to provide a modular proof of the extended deficiency one theorem, Theorem~\ref{thm:def}, and hence of the classical deficiency one theorem.

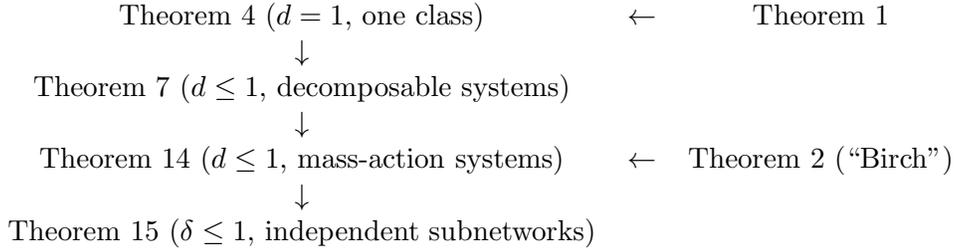
\begin{center}
\begin{tabular}{ccc}
\blue{\underline{Dependency} ($d$) and \underline{deficiency} ($\delta$) results:} \\ \\
Theorem~\ref{thm:one_class} ($d=1$, one class)
& $\leftarrow$ & Theorem~\ref{thm:previous} \\
$\downarrow$ \\
Theorem~\ref{thm:main} ($d_j\le1$, decomposable systems) \\
$\downarrow$ \\
Theorem~\ref{thm:main:mass-action} ($d_j\le1$, mass-action systems) 
& $\leftarrow$ & Theorem~\ref{thm:birch} (``Birch'') \\
$\downarrow$ \\
\blue{Corollary~\ref{cor:main:mass-action} ($d_j\le1$, mass-action, simplified)} \\
$\downarrow$ \\
Theorem~\ref{thm:def} ($\delta_j\le1$, independent subnetworks) 
& $\leftarrow$ & \blue{Lemma~\ref{lem:salt} (``Second salt thm'')} \\
$\downarrow$ \\
\blue{``Classical'' theorem ($\delta_j\le1$, linkage classes)} \\
\end{tabular}
\end{center}

% ========= ========= ========= ========= ========= ========= ========= =========

\subsubsection*{Organization of the work}

In Section~\ref{sec:previous}, we recall the relevant geometric objects corresponding to parametrized systems of generalized polynomial equations (with classes) including the coefficient polytope and the monomial dependency and difference subspaces. 
In Section~\ref{sec:dependency_one}, we provide sufficient conditions for the unique existence of a solution on the coefficient polytope,
along with an equivalent condition for its existence. 
In Section~\ref{sec:reaction_networks},
we introduce basic notions for reaction networks with mass-action kinetics,
and in Section~\ref{sec:network_decomposition},
we decompose a reaction network into subnetworks (as a pre-processing step)
and finally into independent subnetworks.
In Section~\ref{sec:main}, we present our main result, the dependency one theorem for mass-action systems,
and in Section~\ref{sec:deficiency_one}, we derive the extended deficiency one theorem as a special case.
Finally, in Section~\ref{sec:examples}, we provide three examples
where the conditions of the deficiency one theorem are not satisfied,
but the dependency one theorem can be applied (to conclude the unique existence of a positive equilibrium within every stoichiometric compatibility class for some or all rate constants).

\subsubsection*{Notation}
For vectors $x, y \in \R^n$,
we denote their scalar product by $x \cdot y \in \R$
and their component-wise (Hadamard) product by $x \circ y \in \R^n$. 
We denote the vector with all entries equal to one by $\ones_n \in \R^n$
and the identity matrix by $\id_n \in \R^{n \times n}$.

We denote the positive (non-negative) real numbers by $\R_>$ $(\R_\ge)$.
For $x \in \R^n_>$ and $y \in \R^n$, we define the monomial $x^y = \prod_{i=1}^n (x_i)^{y_i} \in \R_>$;
and for $Y = (y^1,\ldots,y^m) \in \R^{n \times m}$,
we define the vector of monomials $x^Y \in \R^m_>$ via $(x^Y)_j = x^{y^j}$.

For $x \in \R^n$, we define $\e^x  = (e^{x_1},e^{x_2}, \ldots, e^{x_n})^\trans \in \R^n_>$;
and for $x \in \R^n_>$, we define $\ln(x)  = \left( \ln(x_1),\ln(x_2), \ldots, \ln(x_n) \right)^\trans \in \R^n$.

%For $x \in \R^n$, we obtain its sign vector $\sign(x) \in \{-,0,+\}^n$ by applying the sign function componentwise.
% For a subset $S \subseteq \R^n$,
% $
% \sign(S) = \{ \sign(x) \mid x \in S \} \subseteq \{-,0,+\}^n .
% $

% ========= ========= ========= ========= ========= ========= ========= =========

\section{Previous results on \blue{polynomial systems}} \label{sec:previous}

In order to state Theorem~\ref{thm:previous} below
for a parametrized system of generalized polynomial equations,
%$A \, (c \circ x^B)=0$,
we introduce geometric objects and auxiliary matrices as defined in~\cite{MuellerRegensburger2023b,MuellerRegensburger2023a}.

\begin{definition} \label{definition}
\blue{For the positive variable vector $x \in \R_>^n$,
a {\em coefficient matrix} $A \in \R^{n' \times m}$,
an {\em exponent matrix} $B \in \R^{n \times m}$,
and a positive {\em parameter vector} $c \in \R_>^m$,
we define the parametrized system of generalized polynomial equations }
\[ A \, (c \circ x^B)=0 . \]
\begin{enumerate}[(i)]
\item \label{def:coeff_cone} 
We call $C = \ker A \cap \R_>^m$ the \emph{coefficient cone}. Its closure $\overline{C} = \ker A \cap \R_\ge^m $ is a polyhedral cone, called an s-cone (subspace cone) in~\cite{MuellerRegensburger2016}.
As a necessary condition for the existence of solutions,
$C$ must be non-empty.
\item
We assume that
$
A = \begin{pmatrix} A_1 & \ldots & A_\ell \end{pmatrix} \in \R^{n' \times m}
$
with $\ell \ge 1$ blocks $A_j \in \R^{n' \times m_j}$ (and hence $m_1+\ldots+m_\ell = m$) 
such that the kernel of $A$
is the direct product of the kernels of $A_j$,
that is, 
$
\ker A = 
\ker A_1 \times \cdots \times \ker A_\ell .
$
Accordingly, 
$
B = \begin{pmatrix} B_1 & \ldots & B_\ell \end{pmatrix} \in \R^{n \times m}
$
with $\ell$ blocks $B_j \in \R^{n \times m_j}$ and 
$
c^\trans = \begin{pmatrix} (c^1)^\trans & \ldots & (c^\ell)^\trans \end{pmatrix} \in \R^m_>
$ 
with $c^j \in \R^{m_j}_>$.

The decomposition of $\ker A$ induces a partition of the indices $\{1,\ldots,m\}$ into $\ell$ {\em classes}.
In particular, the columns of $B=(b^1,\ldots,b^m)$ % the components of $c$, 
and hence the monomials $x^{b^j}$, $j=1,\ldots,m$
are partitioned into classes.
\item 
We introduce the direct product
$
\Delta = \Delta^{m_1-1} \times \cdots \times \Delta^{m_\ell-1} 
$
of the standard simplices
$
\Delta^{m_j-1} = \{ y \in \R^{m_j}_\ge \mid \ones_{m_j} \cdot y = 1 \} 
$
and define the bounded set
$
P = C \cap \Delta .
$

Clearly,
$
P = P_1 \times \cdots \times P_\ell
$
with
$
P_j = C_j \cap \Delta^{m_j-1} .
$

We call $P$ the {\em coefficient polytope}. In fact, $P$ is a polytope without boundary. Strictly speaking, only its closure $\overline P$ is a polytope.
\item
Let 
$I_m = \begin{pmatrix} \id_{m-1} \\ -\ones_{m-1}^\trans \end{pmatrix} \in \R^{m \times (m-1)}$,
which can be seen as the incidence matrix of a star-shaped graph with vertices $\{1,\ldots,m\}$ and root $m$. 
We introduce the $\ell \times \ell$ block-diagonal (incidence) matrix
\begin{equation*}
I = \begin{pmatrix} I_{m_1} & & 0 \\ & \ddots & \\ 0 & & I_{m_\ell} \end{pmatrix} \in \R^{m \times (m-\ell)}
\end{equation*}
with blocks $I_{m_j} \in \R^{m_j \times (m_j-1)}$
and the ``monomial difference'' matrix
$
M = B \, I \in \R^{n\times(m-\ell)} .
$
Clearly, 
$
M = \begin{pmatrix} B_1 I_{m_1} & \ldots &  B_\ell I_{m_\ell} \end{pmatrix}
$
is generated by taking the differences between the first $m_j-1$ columns of $B_j$ and its last column, for $j=1,\ldots,\ell$,
and hence
$
L = \im M \subseteq \R^n
$
is the sum of the linear subspaces associated with the affine spans of the columns of $B$ in the $\ell$ classes.

We call $L$ the {\em monomial difference subspace}.
Further, we call
$
d = \dim (\ker M)
$
the {\em monomial dependency}.
It can be determined as
$
d = m-\ell-\dim L ,
$
cf.~\cite[Proposition~1]{MuellerRegensburger2023a}.

\item
We introduce the $\ell \times \ell$ block-diagonal ``Cayley'' matrix
\begin{equation*}
J = \begin{pmatrix} \ones_{m_1}^\trans & & 0 \\ & \ddots & \\ 0 & & \ones_{m_\ell}^\trans \end{pmatrix} \in \R^{\ell \times m}
\end{equation*}
with blocks $\ones_{m_j}^\trans \in \R^{1 \times m_j}$
and the matrix 
$
\mathcal{B} = \begin{pmatrix} B \\ J \end{pmatrix} \in \R^{(n+\ell) \times m} .
$

We call 
$
D = \ker \mathcal{B} % I (\ker M) 
\subset \R^m
$
the {\em monomial dependency subspace}.
It records affine dependencies between the columns of $B$ within the $\ell$ classes.

In fact, $\dim D = d$,
cf.~\cite[Lemma~4]{MuellerRegensburger2023a}.
\item 
Finally, we introduce the ``exponentiation'' matrix $E = I M^* \in \R^{m \times n}$, where $M^* \in \R^{(m-\ell) \times n}$ is a generalized inverse of $M$.
\end{enumerate}
\end{definition}

We can now state the main result of our previous work 
\blue{on parametrized systems of polynomial inequalities~\cite{MuellerRegensburger2023a}, instantiated for equations, cf.~\cite{MuellerRegensburger2023b}.}

\begin{theorem}[\cite{MuellerRegensburger2023b}, Theorem 1] \label{thm:previous}
Consider the parametrized system of generalized polynomial equations $A \, (c \circ x^B) = 0$. The solution set $Z_c = \{ x \in \R^n_> \mid A \, (c \circ x^B) = 0 \}$ can be written as
\[
Z_c = \{ (y\, \circ\, c^{-1})^E \mid y \in Y_c \} \circ \e^{L^\perp} ,
\]
where
\[
Y_c = \{y \in P \mid y^z = c^z \text{ for all } z\in D\} 
\]
is the solution set on the coefficient polytope $P$.
\end{theorem}

Theorem~\ref{thm:previous} can be read as follows: 
In order to determine the solution set~$Z_c$, 
first determine the solution set on the coefficient polytope, $Y_c$.
Recall that the coefficient polytope $P$ is determined by the coefficient matrix $A$,
and the dependency subspace $D$ is determined by the exponent matrix $B$
(and the classes).
To a solution $y \in Y_c$,
there corresponds the actual solution $x = (y \circ c^{-1})^E \in Z_c$. 
In fact, if (and only if) $\dim L < n$, 
then $y \in Y_c$ corresponds to an exponential manifold of solutions, $x \circ \e^{L^\perp} \subseteq Z_c$.
Strictly speaking, existence of a unique solution corresponds to $|Y_c|=1$ and $\dim L = n$
(that is, $L^\perp = \{0\}$).

\begin{theorem}[``Birch's theorem''] \label{thm:birch}
Let $x_0,x^*\in\R^n_>$ and let $S\subseteq \R^n$ be a subspace. Then we have
\[
|(x_0 + S) \cap (x^*\circ S^{\perp})|=1 .
\]
\end{theorem}

Theorem~\ref{thm:birch} was originally proved by Birch in 1963~\cite{birch1963maximum}
and reproved by Horn and Jackson~\cite[Lemma 4B]{HornJackson1972} 
in the context of reaction networks with mass-action kinetics.
Motivated by applications, Birch's theorem has been extended to cover generalized mass-action kinetics~\cite{MuellerHofbauerRegensburger2019,CraciunMuellerPanteaYu2019}. 

% ========= ========= ========= ========= ========= ========= ========= =========

\blue{
\begin{running}
We consider a parametrized system of two non-overlapping trinomials in two variables,
\begin{alignat*}{3}
0 &= + \, k_{12} - k_{21} \, x_1 - k_{43} \, x_1^3 &&  \\
0 &= && + k_{45} \, x_1^3 + k_{56} \, x_1^2 x_2 - k_{65} \, x_1 x_2^2 ,
\end{alignat*}
arising from a reaction network.%
\footnote{
The reaction network (with rate constants)
\begin{gather*}
    0 \xleftrightarrows{k_{12}}{k_{21}} X_1 , \quad
    2X_1 \xleftarrow[k_{43}]{} 3 X_1 \xrightarrow{k_{45}} 2X_1 + X_2 \xleftrightarrows{k_{56}}{k_{65}} X_1 + 2X_2 
\end{gather*}
defines a mass-action system with associated ODE
\begin{alignat*}{3}
\dd{x_1}{t} &= +\, k_{12} - k_{21} \, x_1 - k_{43} \, x_1^3 &&- k_{45} \, x_1^3 - k_{56} \, x_1^2 x_2 + k_{65} \, x_1 x_2^2 , \\
\dd{x_2}{t} &= &&+ k_{45} \, x_1^3 + k_{56} \, x_1^2 x_2 - k_{65} \, x_1 x_2^2 .
\end{alignat*}
Reaction networks are formally introduced in Section~\ref{sec:reaction_networks},
where we will continue the example.
}
Equivalently,
\[A \left( c\circ x^B \right) = 0 , \]
where
\begin{gather*}
A= 
\begin{pmatrix} 
1 & -1 & -1 & 0 & 0 &  0 \\
0 & 0 &  0 & 1 & 1 & -1 
\end{pmatrix} , \\
B = 
\begin{pmatrix} 
0 & 1 & 3 & 3 & 2 & 1 \\ 
0 & 0 & 0 & 0 & 1 & 2  
\end{pmatrix} , \\
c = (k_{12}, k_{21}, k_{43}, k_{45}, k_{56}, k_{65})^\trans .
\end{gather*}

Clearly, the coefficient matrix $A = \begin{pmatrix} A_1 & A_2 \end{pmatrix} \in \R^{2 \times 6}$ has two blocks $A_1, A_2 \in \R^{2 \times 3}$ such that $\ker A = \ker A_1 \times \ker A_2$.
%with $\ker A_1 = \ker A_2 = \ker A_*$ and $A_* = \begin{pmatrix} 1 & -1 & -1 \end{pmatrix}$.
Accordingly, the exponent matrix $B = \begin{pmatrix} B_1 & B_2 \end{pmatrix}$ has two blocks 
\[
B_1 = 
\begin{pmatrix} 
0 & 1 & 3\\ 
0 & 0 & 0   
\end{pmatrix} 
\quad \text{and} \quad
B_2 = 
\begin{pmatrix} 
3 & 2 & 1 \\ 
0 & 1 & 2  
\end{pmatrix} ,
\]
that is,
there are two {\em classes} of monomials, $1$, $x_1$, $x_1^3$ and $x_1^3$, $x_1^2 x_2$, $x_1 x_2^2$, respectively.
Further, $c = \binom{c_1}{c_2}$ with $c_1 = (k_{12}, k_{21}, k_{43})^\trans$ and $c_2 = (k_{45}, k_{56}, k_{65})^\trans$.

For the coefficient cone $C = \ker A \cap \R^6_>$, we find $C = C_1 \times C_2$ with $C_j = \ker A_j \cap \R^3_>$,
and for the coefficient polytope $P = C \cap \Delta$,
where $\Delta = \Delta^2 \times \Delta^2$ is a product of the standard simplices  $\Delta^2 = \{ y \in \R^3_\ge \mid \ones_3 \cdot y = 1 \}$,
we find
\[ P = P_1 \times P_2 \]
with $P_j = \ker A_j \cap \R^3_> \cap \Delta^2$ and $\dim P_1 = \dim P_2 = 1$.

By appending a Cayley matrix (rows of ones corresponding to the two blocks) to the exponent matrix $B$, we obtain
\[
\mathcal{B} = 
\begin{pmatrix}
0 & 1 & 3 & 3 & 2 & 1 \\ 
0 & 0 & 0 & 0 & 1 & 2 \\
1 & 1 & 1 & 0 & 0 & 0 \\
0 & 0 & 0 & 1 & 1 & 1  
\end{pmatrix} ,
\]
which allows to determine the monomial dependency subspace $D = \ker \mathcal{B}$ and the dependency $d = \dim D$.
In particular, $d = \dim(D) = 2$.
In this example, 
the polynomial system is {\em decomposable}, that is, $D = D_1 \times D_2$ with $D_1 = \im \begin{pmatrix}
2 & -3 & 1 \end{pmatrix}^\trans$ and $D_2 = \im \begin{pmatrix}
1 & -2 & 1 \end{pmatrix}^\trans$ corresponding to the decomposition $\ker A = \ker A_1 \times \ker A_2$.
Clearly, $d_1 = \dim D_1 = d_2 = \dim D_2 = 1$ and hence $d = d_1 + d_2$.

Finally, we obtain the monomial difference matrix $M$ from $B = \begin{pmatrix} B_1 & B_2 \end{pmatrix}$ by taking the differences of the first two columns and the last column within the two blocks,  
\[
M = 
\begin{pmatrix} 
-3 & -2 &  2 &  1 \\ 
 0 &  0 & -2 & -1 
\end{pmatrix} .
\]
This allows to determine the monomial difference subspace $L = \im M$.
In fact, $L = \R^2$.
\end{running}
}

\section{Dependency one systems} \label{sec:dependency_one}

First, we consider systems with one class,
second, we consider decomposable systems.

\subsection{One class}

For one class, we consider $d=\dim P =1$. (The case $d=\dim P =0$ is trivial.)

\begin{definition} \label{definition_one_class}
For a parametrized system of generalized polynomial equations $A \, ( c \circ x^B ) = 0$ 
with 
$\ker A \cap \R^m_> \neq \emptyset$,
one class, one-dimensional coefficient polytope, and monomial dependency one,
let $y^1, y^2 \in {(\ker A \cap \R^m_\ge)}$ be the two vertices of the coefficient polytope,
let $q = (y^1 - y^2) \circ (y^1 + y^2)^{-1} \in \R^m$,
and assume that $1 = q_1 \ge \cdots \ge q_m = -1$
(after reordering of the index set $\{1,\ldots,m\}$).
Further,
let $I_1, \ldots, I_\omega \subset \{1, \ldots,m\}$ be $\omega$ equivalence classes
corresponding to equal (consecutive) components of~$q$,
and let $\bar q \in \R^\omega$ with $\bar q_i = q_{i'}$ for $i' \in I_i$
be the vector of different~$q$'s.
Finally,
let $b \in \R^m$ with $\im b = \ker (\mathcal{B})$,
and let $\bar b \in \R^\omega$ with $\bar b_i = \sum_{i' \in I_i} b_{i'}$
be the vector of lumped~$b$'s.
\end{definition}

\begin{theorem}[$d=1$, one class] \label{thm:one_class}
Let $A \, ( c \circ x^B ) = 0$ be a parametrized system of generalized polynomial equations 
with
$\ker A \cap \R^m_> \neq \emptyset$, 
one class, one-dimensional coefficient polytope ($\dim P = 1$), and monomial dependency one ($d=1$).
Then, 
$|Y_c| = 1$ for all~$c$
if
\begin{center}
$\displaystyle \sum_{i'=1}^{i} \bar b_{i'} \ge 0$  
for all $i=1,\ldots,\omega-1$
(or ``$\le0$'' for all $i$)
and $\bar b_1 \cdot \bar b_\omega < 0$.
\end{center}
\end{theorem}
\begin{proof} 
Let $\hat{y} = \frac{y^1 - y^2}{2}$, $\bar{y} = \frac{y^1 + y^2}{2} > 0$, and hence $q = \hat{y}\circ {\bar{y}}^{-1}$. 
Every $y \in P$ can be written as $y = \bar{y} + t \hat{y}$ with $t \in (-1,1)$,
and the binomial condition $y^{b} = c^b$ for $y \in P$ 
can be written as $(\bar{y} + t \hat{y})^b = c^b$ for $t \in (-1,1)$
or, after division by $\bar y$, as $f(t) := (1 + t q)^b = c^b \, \bar{y}^{-b} =: c^*$. 
After considering equal components of $q$,
\[ % label{eq:f(t)}
f(t) = \prod_{i=1}^m (1+tq_i)^{b_i} = \prod_{i=1}^{\omega} (1+t\bar q_i)^{\bar b_i} .
\]
Now, let $\bar b_1 \cdot \bar b_\omega < 0$, in particular, $\bar b_1 > 0$ and $\bar b_\omega < 0$.
(The other case is analogous.)
Clearly, $\bar q_1 = 1$ implies $f(-1) \to 0$, and $\bar q_\omega = -1$ implies $f(1) \to \infty$. 
By continuity, there is a solution to $f(t) = c^*$ for all $c^*$ and hence for all~$c$.
That is, $|Y_c|\ge 1$ for all~$c$.

Moreover, $f'(t) = f(t) \, h(t)$ with
\[ % \label{eq:main_func_derivative}
h(t) = \sum_{i=1}^\omega \bar b_i \, \frac{\bar q_i}{1 + t \bar q_i} 
= \sum_{i=1}^{\omega-1} \left( \sum_{i'=1}^i \bar b_{i'} \right) 
\underbrace{\left( \frac{\bar q_i}{1 + t \bar q_i} - \frac{\bar q_{i+1}}{1 + t \bar q_{i+1}} \right)}_{>0} .
\]

Now, also let $\sum_{i'=1}^{i} \bar b_{i'} \ge 0$
for all $i=1,\ldots,\omega-1$
(or ``$\le0$'' for all $i$).
Altogether, this implies $f'(t)> 0$ (or $f'(t)< 0$).
That is, $|Y_c| \le 1$ for all~$c$.
\end{proof}

Recently, \blue{the existence of a unique, nondegenerate} solution to a parametrized system of generalized polynomial equations (with arbitrary dependency and using one class) 
\blue{for all parameters)}
has been characterized using Hadamard's Global Inversion Theorem~\cite{DeshpandeMueller2024}.

\blue{Notably,}
existence (without uniqueness) \blue{of a solution} on the coefficient polytope for all parameters can be characterized.

\begin{theorem}\label{thm:existence}
Let $A \, ( c \circ x^B ) = 0$ be a parametrized system of generalized polynomial equations 
with
$\ker A \cap \R^m_> \neq \emptyset$, 
one class, one-dimensional coefficient polytope, and monomial dependency one. 
The following statements are equivalent:
\begin{enumerate}[1.]
\item $|Y_c| \ge 1$ for all~$c$.
\item $\bar b_1\cdot \bar b_\omega <0$.
\end{enumerate}
\end{theorem}
\begin{proof}
See Appendix~\ref{app:proof_existence}.
\end{proof}

% ========= ========= ========= ========= ========= ========= ========= =========

\subsection{Decomposable systems} 

\begin{definition} \label{definition_decomposable systems}
For a parametrized system of generalized polynomial equations $A \, (c \circ x^B) = 0$ with 
$\ker A \cap \R^m_> \neq \emptyset$
and $\ell$ classes, 
let $P_j$ be the coefficient polytope and $d_j$ be the monomial dependency 
of the subsystem $A_j \, (c^j \circ x^{B_j}) = 0$, $j=1,\ldots,\ell$.
If $\dim P_j=1$, let $y^{j,1}, y^{j,2} \in {(\ker A_j \cap \R^{m_j}_\ge)}$ be the two vertices of $P_j$, 
let $q^j = (y^{j,1} - y^{j,2}) \circ (y^{j,1} + y^{j,2})^{-1} \in \R^{m_j}$,
and assume that $1 = q^j_1 \ge q^j_2 \ge \cdots \ge q^j_{m_j} = -1$
(after reordering of the index set $\{1,\ldots,m_j\}$).
Further, let $I^j_1, I^j_2, \ldots, I^j_{\omega_j} \subset \{1, \ldots, m_j\}$ be $\omega_j$ equivalence classes 
corresponding to equal (consecutive) components of~$q^j$.
If $d_j=1$, let $b^j \in \R^{m_j}$ with $\im b^j = \ker (\mathcal{B}_j)$,
and let $\bar b^j \in \R^{\omega_j}$ with $\bar b^j_i = \sum_{i' \in I_i} b^j_{i'}$ be the vector of lumped $b^j$'s.
\end{definition}

\begin{theorem}[$d_j \le 1$] \label{thm:main}
Let $A \, (c \circ x^B) = 0$ be a parametrized system of generalized polynomial equations 
with $\ell$ classes
that fulfills the following conditions:

\begin{enumerate}[(i)]
\item 
$\ker A \cap \R^m_> \neq \emptyset$.
\item 
$d = d_1 + \cdots + d_{\ell}$.
\item 
For every (class) $j=1,\ldots,\ell$,
\begin{itemize}
    \item 
    $d_j = \dim P_j \le 1$. 
    \item 
    If $d_j=1$, then
    \begin{itemize}
        \item 
        $\displaystyle \sum_{i'=1}^{i} \bar b^j_{i'} \ge 0$ for all $i=1,\ldots,\omega_j-1$
        (or ``$\le0$'' for all $i$)
        \item
        $\bar b^j_1 \cdot \bar b^j_{\omega_j} < 0$.
    \end{itemize}
\end{itemize}
\end{enumerate}
Then, $|Y_c| = 1$ for all~$c$.
\end{theorem}

\begin{proof}
By (i), the coefficient polytope $P = P_1 \times \cdots \times P_\ell$ is non-empty.
By (ii) and \cite[\blue{Section~3.2}]{MuellerRegensburger2023a}, the system is \blue{{\em decomposable}}.
By \cite[Proposition~10]{MuellerRegensburger2023a},
$Y_c = Y_{c,1} \times \cdots \times Y_{c,{\ell}}$,
that is, the solution set on the coefficient polytope is a direct product.
In particular, \blue{$|Y_c| = |Y_{c,1}| \cdot \, \cdots \, \cdot |Y_{c,\ell}|$.}
By the first item in (iii), $d_j = \dim P_j \le 1$.

Case $d_j = \dim P_j = 0$:
$P_j$ is a point, and there is no binomial condition on $P_j$.
Hence, $|Y_{c,j}|=1$. %for all~$c$. 

Case $d_j = \dim P_j = 1$: 
By the second item in (iii),
$\sum_{i'=1}^{i} \bar b^j_{i'} \ge 0$ for all $i=1,\ldots,\omega_j-1$
(or ``$\le0$'' for all $i$) and $\bar b^j_1 \cdot \bar b^j_{\omega_j} < 0$.
By Theorem~\ref{thm:one_class} for the subsystem $A_j \, (c^j \circ x^{B_j}) = 0$,
$|Y_{c,j}|=1$. 

Altogether, \blue{$|Y_c| = |Y_{c,1}| \cdot \, \cdots \, \cdot |Y_{c,\ell}| = 1$},
and all implications hold for all~$c$.
\end{proof}

% ========= ========= ========= ========= ========= ========= ========= =========

\section{Reaction networks} \label{sec:reaction_networks}

We recall basic notions for reaction networks with mass-action kinetics~\cite{adleman2014mathematics,voit2015150,yu2018mathematical,gunawardena2003chemical} \blue{following~\cite{MuellerRegensburger2014}.
Note that we use index notation (rather than pure cardinality notation) in the definition of matrices,
thereby avoiding arbitrary orderings of the underlying sets.}

A reaction network $(G,y)$ 
is given by a simple directed graph $G=(V,E)$
\blue{with a finite set of vertices $V$ and edge set $E \subseteq V \times V$
together with an injective} map ${y \colon V \to \R^n}$
\blue{(a matrix $Y \in \R^{n \times V}$).}
Every vertex $i \in V$
is labeled with a {\em (stoichiometric) complex} $y(i) \in \R^n$,
and every edge $(i \to i') \in E$ represents a {\em reaction} ${y(i) \to y(i')}$.
\blue{The restriction of the map $y$ (the matrix $Y$) to the set of source vertices $V_s \subseteq V$ is denoted by $y_s \colon V_s \to \R^n$ (respectively, $Y_s \in \R^{n \times V_s}$).}
If all components of the graph are strongly connected,
the network is called {\em weakly reversible}.

\begin{remark}
In the classical definition of a reaction network,
complexes (like reactions) are primary objects and correspond to the vertices (and edges) of the induced complex-reaction graph.
Such reaction networks can also be represented as Euclidean-embedded graphs~\cite{craciun2015toric}.
See also \cite{craciun2019polynomial,craciun2020endotactic}.
\end{remark} 

A {\em mass-action system} $(G_k,y)$ is given by a reaction network $(G,y)$
and positive edge labels $k \in \R^E_>$.
Every edge/reaction $(i \to i') \in E$ is labeled with a {\em rate constant} $k_{i \to i'} > 0$.

The associated ODE system for the \blue{non-negative} {\em concentrations} $x \in \R^n_\ge$ (of $n$ chemical species) is given by
\begin{equation} \label{eq:ode}
	\dd{x}{t} 
	= \sum_{(i \to i') \in E} k_{i \to i'} \, x^{y(i)} \big( y(i')-y(i) \big) .
\end{equation}
The sum ranges over all reactions, 
and every summand is a product of the {\em reaction rate} $k_{i \to i'} \, x^{y(i)}$, involving a monomial $x^{y} = \prod_{j=1}^n (x_j)^{y_j}$ determined by the stoichiometric complex of the reactant,
and the {\em reaction vector} $y(i')-y(i)$ given by the stoichiometric complexes of product and reactant.

Let $I_E \in \{-1,0,1\}^{V \times E}$ and $I_{E,s} \in \{0,1\}^{V_s \times E}$ be the incidence and source matrices of the digraph~$G$, respectively,
and $R_k = I_E \diag(k) (I_{E,s})^\trans \in \R^{V \times V_s}$
be the rectangular ``Laplacian matrix''.
(For details on the index notation, see Appendix~\ref{app:index}.)
Further, 
\blue{recall the (source) complex matrices,} $Y \in \R^{n \times V}$ and $Y_s \in \R^{n \times V_s}$,
and let $N = Y I_E \in \R^{n \times E}$ be the stoichiometric matrix.
Then, 
the right-hand-side of \eqref{eq:ode} can be written in matrix form and decomposed into stoichiometric and graphical contributions,
\begin{equation} \label{eq:ode:matrices}
\begin{aligned}
	\dd{x}{t}
    &= \underbrace{Y I_E}_N \left( k \circ x^{Y_s I_{E,s}} \right)  
	= Y \underbrace{I_E \diag(k) (I_{E,s})^\trans}_{R_k} x^{Y_s} \\
    &= \Gamma_k \, x^{Y_s} .
    \end{aligned}
\end{equation}
\blue{In the last step,
we have introduced the {\em kinetic matrix} $\Gamma_k = Y R_k \in \R^{n \times V_s}$,
in analogy to the stoichiometric matrix $N = Y I_E$.}

\begin{remark}
Traditionally,
one uses the source matrix $I'_{E,s} \in \{0,1\}^{V \times E}$ which involves all vertices (not just the source vertices),
and one obtains
\[
\dd{x}{t} = Y I_E \left( k \circ x^{Y I'_{E,s}} \right) 
= Y \underbrace{I_E \diag(k) (I'_{E,s})^\trans}_{\mathcal{L}_k} x^{Y}
\]
with the (square) Laplacian matrix $\mathcal{L}_k \in \R^{V \times V}$.
This formulation can be misleading
since columns of $\mathcal{L}_k$ corresponding to non-source vertices are zero,
and, after multiplication, non-source monomials do not appear in $\dd{x}{t} = Y \blue{\mathcal{L}_k} \, x^{Y}$.
Indeed,
$R_k$ arises from $\mathcal{L}_k$ by deleting zero columns,
and $\im R_k = \im \mathcal{L}_k$. % \subseteq \im I_E$.
\end{remark}

The change over time lies in the {\em kinetic subspace} $K = \im (Y R_k)$
and further in the stoichiometric subspace {\em stoichiometric subspace} $S = \im (Y I_E)$, 
\[
\dd{x}{t} \in K \subseteq S .
\]
Hence, trajectories are confined to cosets of $K$ and $S$, respectively,
that is, $x(t) \in x(0)+K \subseteq x(0) + S$.
For positive $x' \in \R^n_>$, the sets $(x'+K) \subseteq (x'+S) \cap \R^n_>$ are called {\em kinetic} and {\em stoichiometric compatibility classes}, respectively.

Finally, we introduce non-negative integer characteristics of a graph or a reaction network.
In particular, let $l$ be the number of {\em components} of $G$ (not to be confused with $\ell$ defined below),
let $t$ be the number of {\em \blue{absorbing} strong components},
and let $t'$ be the number of \blue{non-singleton absorbing} strong components.
It is well known that 
\[
\dim (\ker I_E^\trans) = l
\quad \text{and} \quad
\dim (\ker \mathcal{L}_k) = t ,
\]
which further implies
\[
\dim(\im I_E) = |V| - l
\quad \text{and} \quad
\dim(\im R_k) = \dim(\im \mathcal{L}_k) = |V| - t .
\]
Analogously, 
\[
\dim (\ker R_k) = t' .
\]

Most importantly, the (stoichiometric) {\em deficiency} is given by
\[
\delta 
= \dim(\ker Y \cap \im I_E) 
= |V| - l - \dim(S) .
\]

From the facts above,
it can be easily shown that
\[
t=l \quad\implies\quad K=S \quad\implies\quad \delta \ge t-l .
\]
Mass-action systems with $K \neq S$ can be ``pathological''.
The assumption $t=l$ in the classical deficiency one theorem~\cite{Feinberg1987} rules out such systems.
In the extended deficiency one theorem~\cite{Feinberg1995a} (for independent subnetworks),
this assumption is missing.

\blue{For a summary of the notation introduced in this section, see Table~\ref{tab}(a).
For an illustration, we continue the example from Section~\ref{sec:previous}.}

\blue{
\begin{running}[continued]
We return to the reaction network (with rate constants)
\begin{gather*}
    0 \xleftrightarrows{k_{12}}{k_{21}} X_1 , \quad
    2X_1 \xleftarrow[k_{43}]{} 3 X_1 \xrightarrow{k_{45}} 2X_1 + X_2 \xleftrightarrows{k_{56}}{k_{65}} X_1 + 2X_2 
\end{gather*}
or, equivalently, the ``embedded'' graph
\begin{center}
\begin{tikzpicture}[scale=2]
    \draw [step=1, gray, very thin] (0,0) grid (3,2);
\draw [ ->, black!70!white] (0,0)--(3.25,0);
\draw [ ->, black!70!white] (0,0)--(0,2.25);

\begin{scope}[shift={(0,0)}]

\node[circle, fill=blue, inner sep=1pt, outer sep=5pt] (A) at (0,0) {};
\node[circle, fill=blue, inner sep=1pt, outer sep=5pt] (B) at (1,0) {};
\node[circle, fill=blue, inner sep=1pt, outer sep=5pt] (C) at (2,0) {};
\node[circle, fill=blue, inner sep=1pt, outer sep=5pt] (D) at (3,0) {};
\node[circle, fill=blue, inner sep=1pt, outer sep=5pt] (E) at (2,1) {};
\node[circle, fill=blue, inner sep=1pt, outer sep=5pt] (F) at (1,2) {};

\node[inner sep=0,outer sep=1] (0) at (0,0) {};

\draw[arrows={-stealth},very thick,blue,transform canvas={yshift=2pt}] (A) to (B);
\draw[arrows={-stealth},very thick,blue,transform canvas={yshift=-2pt}] (B) to (A);
\draw[arrows={-stealth},very thick,blue,transform canvas={yshift=0pt}] (D) to (C);
\draw[arrows={-stealth},very thick,blue,transform canvas={xshift=0pt}] (D) to (E);
\draw[arrows={-stealth},very thick,blue,transform canvas={xshift=1.4pt,yshift=1.4pt}] (E) to (F);
\draw[arrows={-stealth},very thick,blue,transform canvas={xshift=-1.4pt,yshift=-1.4pt}] (F) to (E);

\node [below left] (A) at (A) {$1$};
\node [below] (B) at (B) {$2$};
\node [below] (C) at (C) {$3$};
\node [below right] (D) at (D) {$4$};
\node [above right] (E) at (E) {$5$};
\node [above left] (F) at (F) {$6$};

%\node [below] (0) at (0) {$0$};

\end{scope}

% \begin{scope}[shift={(1.5,4)}]
%\node at (0,0) 

%{$\begin{aligned} 
%\frac{dx}{t} &=  k_1 xy + k_2x^2y^2  - k_3x^3y^2 + k_4x^3y^2 + k_5x^4y \\
%\frac{dy}{dt} &= k_1 xy - k_4 x^3y^2 
%\end{aligned}$};

\end{tikzpicture}
\end{center}
In our definition, the reaction network $(G,y)$ is given by the graph $G=(V,E)$ with vertex set $V=\{1,2,3,4,5,6\}$ and edge set $E=\{12,21,43,45,56,65\}$
(where we write $ij$ short for $i\to j$)
and the map $y \colon V \to \R^n$ with $n=2$ and $y(1)=\binom{0}{0}, y(2) = \binom{1}{0}$, $y(3)=\binom{2}{0}, y(4) = \binom{3}{0}$, $y(5)=\binom{2}{1}, y(6) = \binom{1}{2}$ (which labels vertices with complexes).
The mass-action system $(G_k,y)$ further labels edges with positive rate constants $k = (k_{12}, k_{21}, k_{43}, k_{45}, k_{56}, k_{65})^\trans \in \R^E_>$.
Note that the set of source vertices is $V_s=\{1,2,4,5,6\}$.

First, the $V \times E$ incidence and $V_s \times E$ source matrices (of the graph) are given by
\small
\[
I_E = 
\bordermatrix{& \gray{12} & \gray{21} & \gray{43} & \gray{45} & \gray{56} & \gray{65} \cr
\gray{1} &-1 & 1 & 0 & 0 & 0 & 0 \cr
\gray{2} & 1 &-1 & 0 & 0 & 0 & 0 \cr
\gray{3} & 0 & 0 & 1 & 0 & 0 & 0 \cr
\gray{4} & 0 & 0 &-1 &-1 & 0 & 0 \cr
\gray{5} & 0 & 0 & 0 & 1 &-1 & 1 \cr
\gray{6} & 0 & 0 & 0 & 0 & 1 &-1
}
\quad \text{and} \quad
I_{E,s} = 
\bordermatrix{& \gray{12} & \gray{21} & \gray{43} & \gray{45} & \gray{56} & \gray{65} \cr
\gray{1} & 1 & 0 & 0 & 0 & 0 & 0 \cr
\gray{2} & 0 & 1 & 0 & 0 & 0 & 0 \cr
\gray{4} & 0 & 0 & 1 & 1 & 0 & 0 \cr
\gray{5} & 0 & 0 & 0 & 0 & 1 & 0 \cr
\gray{6} & 0 & 0 & 0 & 0 & 0 & 1
} 
\]
\normalsize
with resulting $V \times V_s$ (rectangular) Laplacian matrix
\small
\[
R_k = I_E \diag(k) (I_{E,s})^\trans = 
\bordermatrix{& \gray{1} & \gray{2} & \gray{4} & \gray{5} & \gray{6} \cr
\gray{1} &-k_{12} & k_{21} &             0 &      0 &      0 \cr
\gray{2} & k_{12} &-k_{21} &             0 &      0 &      0 \cr
\gray{3} &      0 &      0 & k_{43}        &      0 &      0 \cr
\gray{4} &      0 &      0 &-k_{43}-k_{45} &      0 &      0 \cr
\gray{5} &      0 &      0 &        k_{45} &-k_{56} & k_{65} \cr
\gray{6} &      0 &      0 &             0 & k_{56} &-k_{65}
} .
\]
\normalsize

Second, 
the $n \times V$ complex and $n \times V_s$ source complex matrices (arising from the underlying map) are given by
\small
\[
Y = 
\bordermatrix{& \gray{1} & \gray{2} & \gray{3} & \gray{4} & \gray{5} & \gray{6} \cr
& 0 & 1 & 2 & 3 & 2 & 1 \cr
& 0 & 0 & 0 & 0 & 1 & 2
}
\quad \text{and} \quad
Y_s = 
\bordermatrix{& \gray{1} & \gray{2} & \gray{4} & \gray{5} & \gray{6} \cr
& 0 & 1 & 3 & 2 & 1 \cr
& 0 & 0 & 0 & 1 & 2
} .
\]
\normalsize

Third, the $n \times E$ stoichiometric and $n \times V_s$ kinetic matrices are given by
\small
\[
N = Y I_E =
\bordermatrix{& \gray{12} & \gray{21} & \gray{43} & \gray{45} & \gray{56} & \gray{65} \cr
& 1 &-1 &-1 &-1 &-1 & 1 \cr
& 0 & 0 & 0 & 1 & 1 &-1
}
\]
\text{and}
\[
\Gamma_k = Y R_k = 
\bordermatrix{& \gray{1} & \gray{2} & \gray{4} & \gray{5} & \gray{6} \cr
& k_{12} &-k_{21} &-k_{43}-k_{45} &-k_{56} & k_{65} \cr
& 0      & 0      & k_{45}        & k_{56} &-k_{65}
} .
\]
\normalsize

They define the stoichiometric and kinetic subspaces, $S = \im N = \R^2$ and $K= \im \Gamma_k = \R^2$ (for all $k$).

The network has $l=2$ components (with vertex sets $\{1,2\}$ and $\{3,4,5,6\}$) and $t=3$ absorbing strong components (with vertex sets $\{1,2\}$, $\{3\}$, and $\{5,6\}$), in particular, $t\neq l$.
The resulting deficiency is $\delta = |V|-l-\dim S = 6 - 2 - 2 = 2$.
Further,
we note that the deficiencies of the two components are $\delta_1 = 2 - 1 - 1 = 0$ and $\delta_2 = 4 - 1 - 2 = 1$, and hence $\delta \neq \delta_1+\delta_2$.

\end{running}
}

% ========= ========= ========= ========= ========= ========= ========= =========

\section{Network decomposition} \label{sec:network_decomposition}

First, we decompose a reaction network/mass-action system and the associated ODE.
Then, in Subsection~\ref{sec:equil},
we formulate the resulting polynomial equations for positive equilibria. Only in Subsection~\ref{sec:indep},
we assume independent subnetworks (in the sense of~\cite{Feinberg1987}) and \blue{address} the decomposability of the polynomial equations (in the sense of~\cite{MuellerRegensburger2023a}).

In order to decompose the right-hand side of the ODE of a mass-action system,
\begin{equation} \label{eq:ODE:start}
\dd{x}{t} 
= Y I_E \left( k \circ x^{Y_s I_{E,s}} \right) ,
\end{equation}
we proceed in three steps.

(i) We assume that the edge set is partitioned into $\ell$ disjoint subsets,
$E = E^1 \dot\cup \cdots \dot\cup E^\ell$.
(For the moment, the partition is arbitrary.
In Subsection~\ref{sec:indep}, 
it will arise from writing the kernel of $N = Y I_E$
%or $\Gamma_k = Y R_k$
as a direct product.)
The partition induces the subgraphs $G^j = (V^j,E^j)$, $j = 1, \ldots, \ell$,
with vertex sets $V^j$, source vertex sets $V_s^j$, and non-source vertex sets $V_\textit{ns}^j$.
Note that the vertex sets need not be disjoint.
Further note that
\blue{absorbing strong components are either non-source vertices or non-singleton strong components.
Let} $t_j$ denote the number of \blue{absorbing} strong components (of the subgraph $G^j$)
and $t'_j$ denote the number of \blue{non-singleton absorbing} strong components. Clearly,
$t_j = t'_j + |V^j_\textit{ns}|$.
The subgraph $G^j$ has
incidence and source matrices
\[
I_E^j \in \{-1,0,1\}^{V^j \times E^j} , \quad
I_{E,s}^j \in \{0,1\}^{V_s^j \times E^j} .
\]

The corresponding subnetwork $(G^j,y^j)$ has
complex matrix $Y^j \in \R^{n \times V^j}$
and source complex matrix $Y_s^j \in \R^{n \times V_s^j}$.
(All matrices for a subgraph are submatrices of the corresponding matrices for the full graph.)
Clearly,
\[
Y I_E 
= 
\begin{pmatrix}
Y^1 I_E^1 & \ldots & Y^\ell I_E^\ell
\end{pmatrix} 
\quad \text{and} \quad
Y_s \, I_{E,s} 
= 
\begin{pmatrix}
Y_s^1 I_{E,s}^1 & \ldots & Y_s^\ell I_{E,s}^\ell
\end{pmatrix} 
.
\]
The vector of parameters $k \in \R^E_>$ is partitioned accordingly,
\[
k = 
\begin{pmatrix}
k^1 \\ \vdots \\ k^\ell
\end{pmatrix}
%\in \R_>^{E^1 \dot\cup \cdots \dot\cup E^\ell}
\]
with $k^j \in \R_>^{E^j}$,
and the rectangular Laplacian matrix
of the mass-action system $(G_k^j,y^j)$
is given by
\[
R_k^j = I_E^j \diag(k^j) (I_{E,s}^j)^\trans \in \R^{V^j \times V_s^j} .
\]

(ii) In turn,
we define ``combined'' block(-diagonal) matrices from the matrices of the subgraphs.
Since the vertex sets need not be disjoint, 
we form the {\em disjoint unions} of (source) vertex sets, $V^\sqcup = V^1 \sqcup \cdots \sqcup V^\ell$ and $V_s^\sqcup = V_s^1 \sqcup \cdots \sqcup V_s^\ell$.
We introduce the combined incidence and source matrices
\[
%\small
I^*_E
=
\begin{pmatrix}
I_E^1 & & 0 \\
& \ddots & \\
0 & & I_E^\ell 
\end{pmatrix} 
\in \{-1,0,1\}^{V^\sqcup \times E}
, \quad
I^*_{E,s}
=
\begin{pmatrix}
I_{E,s}^1 & & 0 \\
& \ddots & \\
0 & & I_{E,s}^\ell 
\end{pmatrix} 
\in \{0,1\}^{V_s^\sqcup \times E} 
\]
and the combined rectangular Laplacian matrix
\[
R^*_k 
=
I^*_E \diag(k) (I^*_{E,s})^\trans
=
\begin{pmatrix}
R_k^1 & & 0 \\
& \ddots & \\
0 & & R_k^\ell 
\end{pmatrix} 
\in \R^{V^\sqcup \times V_s^\sqcup} .
\]
The (source) complex matrices can be combined accordingly,
\[
Y^* 
=
\begin{pmatrix}
Y^1 & \ldots & Y^\ell
\end{pmatrix}
\in \R^{n \times V^\sqcup}
, \quad
Y_s^* 
=
\begin{pmatrix}
Y_s^1 & \ldots & Y_s^\ell
\end{pmatrix}
\in \R^{n \times V_s^\sqcup} .
\]
On the one hand,
$
Y^* I^*_E = Y I_E .
$
On the other hand,
$
Y_s^* = Y_s \, I^*_{V,s}
$
with a matrix $I^*_{V,s} \in \{0,1\}^{V_s \times V_s^\sqcup}$
that assigns to every source vertex of a subgraph the corresponding source vertex of the full graph.

(iii)
Now, we return to the ODE~\eqref{eq:ODE:start}.
Like the matrix $Y I_E$,
the vector of monomials ${x^{Y_s I_{E,s}} \in \R^E_\ge}$ has $\ell$ blocks and can be decomposed as
\begin{align*}
x^{Y_s I_{E,s}} 
&= 
\begin{pmatrix}
x^{Y_s^1 I_{E,s}^1} \\
\vdots \\
x^{Y_s^\ell I_{E,s}^\ell} 
\end{pmatrix} 
=
\begin{pmatrix}
(I_{E,s}^1)^\trans x^{Y_s^1} \\
\vdots \\
(I_{E,s}^\ell)^\trans x^{Y_s^\ell}
\end{pmatrix}
\\
&= 
\begin{pmatrix}
(I_{E,s}^1)^\trans & & 0 \\
& \ddots & \\
0 & & (I_{E,s}^\ell)^\trans 
\end{pmatrix} 
\begin{pmatrix}
x^{Y_s^1} \\
\vdots \\
x^{Y_s^\ell} 
\end{pmatrix} 
=
(I^*_{E,s})^\trans x^{Y^*_s} .
\end{align*}

Using the definitions above,
we decompose the right-hand side of the ODE as
\begin{align*}
\dd{x}{t} 
&= Y I_E \left( k \circ x^{Y_s I_{E,s}} \right) 
= Y I_E \diag(k) \, x^{Y_s I_{E,s}} \\
&= Y^* I^*_E \diag(k) (I^*_{E,s})^\trans x^{Y^*_s} 
= Y^* R^*_k \, x^{Y^*_s} .
\end{align*}

Finally, we introduce \blue{the kinetic matrices}
\[
\Gamma_k^j = Y^j R_k^j \in \R^{n \times V_s^j} , \quad j=1,\ldots,\ell,
\]
and \blue{the combined kinetic matrix}
\[ 
\Gamma_k^*
= Y^* R_k^*
= \begin{pmatrix} \Gamma_k^1 & \ldots & \Gamma_k^\ell \end{pmatrix}
= \begin{pmatrix} Y^1 R_k^1 & \ldots & Y^\ell R_k^\ell \end{pmatrix} \in \R^{n \times V_s^\sqcup} 
\]
and summarize the decomposition as
\begin{equation} \label{eq:ODE:end}
\dd{x}{t} 
= 
\Gamma_k^* \, x^{Y^*_s} \\
= \sum_{j =1}^\ell \Gamma_k^j \, x^{Y_s^j}.
\end{equation}

Before we proceed with the treatment of the resulting polynomial equations,
we introduce linear subspaces associated with a subnetwork $(G^j,y^j)$ or a mass-action system $(G^j_k,y^j)$, $j=1,\ldots,\ell$.
As above,
we define the stoichiometric and kinetic subspaces
$S_j = \im (Y^j I_E^j)$ and $K_j = \im(Y^j R^j_k)$,
where $K_j \subseteq S_j$ since $\im R^j_k \subseteq \im I_E^j$, 
\blue{and further the combined kinetic subspace $K^* = \im (Y^* R_k^*) = \sum_j K_j$.
%
% Clearly, 
% $S = \sum_{j=1}^\ell S_j$. % and $L = \sum_{j=1}^\ell L_j$.
% However, for~$K$, there is only an inclusion.
We record the following fact.
}

\begin{fact} \label{fac:K_sum_K} \marginnote{not used}
\blue{Let $(G,y)$ be a reaction network decomposed into $\ell$ subnetworks.
Then, $S = \sum_{j=1}^\ell S_j$
and}
$K \subseteq K^* = \sum_{j=1}^\ell K_j$.
\end{fact}
\begin{proof}
\blue{
Recall $Y I_E = Y^* I^*_E 
= \begin{pmatrix}
Y^1 I_E^1 & \ldots & Y^\ell I_E^\ell
\end{pmatrix} $
and $Y^* R_k^* = \begin{pmatrix} Y^1 R_k^1 & \ldots & Y^\ell R_k^\ell \end{pmatrix}$.
First,
\[
S = \im (Y I_E) = \im (Y^1 I_E^1) + \cdots + \im (Y^\ell I_E^\ell)
= S^1 + \cdots + S^\ell .
\]
Second,}
$K = \im (Y R_k)$ and \blue{$K^* = \im (Y^* R_k^*)$,}
where
\[
Y R_k 
= Y I_E \diag(k) (I_{E,s})^\trans 
\]
and 
\[
Y^* R_k^*
= Y^* I^*_E \diag(k) (I^*_{E,s})^\trans .
\]
Since $Y I_E = Y^* I^*_E$ and $\im (I_{E,s})^\trans \subseteq \im (I^*_{E,s})^\trans$,
we have $\im (Y R_k) \subseteq \im (Y^* R_k^*)$.
\end{proof}

% ========= ========= ========= ========= ========= ========= ========= =========

\subsection{Positive equilibria} \label{sec:equil}

Finally, we consider positive equilibria of the ODE~\eqref{eq:ODE:start}, 
that is,
$x \in \R_>^n$ such that $\dd{x}{t} = 0$.
Clearly,
we can specify them in the form $A \, (c \circ x^B) = 0$ as
\begin{subequations}
\begin{equation} \label{AcBa}
    N \, (k \circ x^{Y_s I_{E,s}}) = 0 ,
\end{equation}
that is, with $A = N \in \R^{n \times E}$, $c = k \in \R_>^E$, $B = Y_s \, I_{E,s} \in \R^{n \times E}$;
equivalently, using the decomposition~\eqref{eq:ODE:end}, 
we can specify them as
\begin{equation} \label{AcBb}
    \Gamma_k^* \, x^{Y^*_s} = 0 , 
\end{equation}
that is, with $A = \Gamma_k^* \in \R^{n \times V_s^\sqcup}$, $c = 1$, and $B = Y^*_s % = Y_s \, I^*_{V,s}
\in \R^{n \times V_s^\sqcup}$.
\end{subequations}

In Equation~\eqref{AcBa},
the coefficient matrix $A=N$ (and hence the classes determined by its kernel) do not depend on the parameters $k$.
However, there may be repeated monomials (within classes) giving rise to trivial dependencies.
In Equation~\eqref{AcBb},
the coefficient matrix $A=\Gamma_k^*$ does depend on $k$,
but repeated monomials are handled via the (rectangular) Laplacian matrix (which also eliminates non-source monomials).
Hence, we will consider Equation~\eqref{AcBb},
but use classes arising from Equation~\eqref{AcBa}.

% ========= ========= ========= ========= ========= ========= ========= =========

\blue{
\subsection{Independent subnetworks and decomposability} \label{sec:indep}
}

Let $(G,y)$ \blue{with $G=(V,E)$} be a reaction network.
In the following, we assume that the partition of the edge set, \blue{$E = E^1 \dot\cup \cdots \dot\cup E^\ell$,} arises
from \blue{a decomposition of $\ker N$ (as a {\em direct product})}
or, equivalently, from \blue{a decomposition of $S=\im N$ (as a {\em direct sum}).}
That is, 
\begin{subequations}
\[
N = Y I_E 
= \begin{pmatrix} N^1 & \ldots & N^\ell \end{pmatrix} 
= \begin{pmatrix} Y^1 I_E^1 & \ldots & Y^\ell I_E^\ell \end{pmatrix}
\]
such that
\begin{equation} \label{eq:prod:N}
\ker N = \ker N^1 \times \cdots \times \ker N^\ell .
\end{equation}
Equivalently, 
\[
S = S_1 \oplus \cdots \oplus S_\ell .
\]
In the terminology of \cite{Feinberg1987},
the \blue{resulting} subnetworks $(G^j,y^j)$ are {\em independent}.
\blue{In the setting of the deficiency one theorem,
we will assume that the subgraphs $G^j$ are connected
(and say that the subnetworks are connected),
see Section~\ref{sec:deficiency_one}.}

\blue{
Now, let $(G_k,y)$ be the corresponding mass-action system
with combined kinetic matrix}
\[
\Gamma_k^*
= Y^* R_k^*
= \begin{pmatrix} \Gamma_k^1 & \ldots & \Gamma_k^\ell \end{pmatrix}
= \begin{pmatrix} Y^1 R_k^1 & \ldots & Y^\ell R_k^\ell \end{pmatrix} .
\]
\blue{By Proposition~\ref{pro:decomp} below,}
\begin{equation} \label{eq:prod:Gamma}
\ker \Gamma_k^* = \ker \Gamma_k^1 \times \cdots \times \ker \Gamma_k^\ell .
\end{equation}
\end{subequations}
\blue{
Equivalently, 
\[
K^* = K_1 \oplus \cdots \oplus K_\ell .
\]
}%
In the terminology of \cite{MuellerRegensburger2023a}
(for the polynomial equations $\Gamma_k^* \, x^{Y_s^*} = 0$), 
the decomposition of $\ker \Gamma_k^* \subseteq \R^{ V_s^\sqcup}$
induces a partition of the source vertex set $V_s^\sqcup = V_s^1 \sqcup \cdots \sqcup V_s^\ell$ into the $\ell$ {\em classes} $V_s^j$.
\blue{It remains to provide a formal argument.}

\blue{
\begin{proposition} \label{pro:decomp}
Let $(G,y)$ be a reaction network decomposed into $\ell$ independent subnetworks and $(G_k,y)$ be the corresponding mass-action system.
%with $G=(V,E)$ and $E = E^1 \dot\cup \cdots \dot\cup E^\ell$.
Then, \eqref{eq:prod:N} implies \eqref{eq:prod:Gamma}.
\end{proposition}
}
\begin{proof}
Consider $\xi \in \R^{V_s^\sqcup}$ with blocks $\xi^j \in \R^{V_s^j}$, $j=1,\ldots,\ell$, 
and assume $\Gamma_k^* \, \xi = \sum_{j=1}^\ell \Gamma_k^j \, \xi^j = 0$.
\blue{Further,} introduce $\alpha^j = \diag(k^j) (I_{E,s}^j)^\trans \xi^j \in \R^{E^j}$. \blue{Then,}
\[
N^j \alpha^j = Y^j I_E^j \diag(k^j) (I_{E,s}^j)^\trans \xi^j = Y^j R_k^j \, \xi^j = \Gamma_k^j \, \xi^j 
\]
\blue{and hence} $\sum_{j=1}^\ell N^j \alpha^j = 0$. By~\eqref{eq:prod:N}, $N^j \alpha^j = \Gamma_k^j \, \xi^j = 0$. That is, \eqref{eq:prod:Gamma}.
\end{proof}

\blue{
To summarize, 
the decomposition of the kernel of the stoichiometric matrix implies a partition of the edge set (and a corresponding decomposition of the network into {\em independent subnetworks})
and further a decomposition of the kernel of the combined kinetic matrix (and a corresponding partition of the source vertex set into {\em classes}).
An algorithm for network decomposition has been developed in~\cite{Hernandez2021} and applied to derive steady states analytically in~\cite{Hernandez2023}.
% As a consequence, the finest decomposition of $\ker N$ 
% can be readily determined (since it does not depend on the parameters $k$),
% and it implies a corresponding decomposition of $\ker \Gamma_k$; 
% however, this need not be the finest one.
}

For the polynomial equations $\Gamma_k^* \, x^{Y_s^*} = 0$,
the monomial dependency and difference subspaces, $D$ and $L$, are given in 
\blue{Definition~\ref{definition} (with $B=Y_s^*$).
Assuming $\ell$ classes,}
we introduce corresponding subspaces $D_j$ and $L_j$,
\blue{$j = 1,\ldots,\ell$.}
In particular,
$L_j = \im (Y_s^j I_{\mathcal{E}_s}^j)$,
where $I_{\mathcal{E}_s}^j \in \{-1,0,1\}^{V_s^j \times \mathcal{E}_s}$ is the incidence matrix of \blue{a star-shaped graph $(V_s^j,\mathcal{E}_s)$ (on the source vertices)} with $|\mathcal{E}_s| = |V_s^j| - 1$.
%(For details on auxiliary graphs, see Appendix~\ref{app:aux}.)

As defined in \blue{\cite[Section~3.2]{MuellerRegensburger2023a},
the} polynomial equations \blue{$\Gamma_k^* \, x^{Y_s^*} = 0$} are {\em decomposable} if 
$D$ \blue{has the decomposition} $D=D_1 \times \cdots \times D_\ell$,
\blue{corresponding to the decomposition of $\ker \Gamma_k^*$;}
equivalently, $d = d_1 + \cdots + d_{\ell}$,
where $d_j = \dim D_j$.
\blue{We provide another characterization.}

% For mass-action systems, the resulting polynomial equations are decomposable.
% (In general, this does not hold for generalized mass-action systems.)

\begin{proposition} \label{pro:decomposable} \marginnote{used in Corollary~\ref{cor:main:mass-action}}
\blue{Let $(G,y)$ be a reaction network decomposed into $\ell$ (independent) 
subnetworks.}
%and $(G_k,y)$ be the corresponding mass-action system with associated ODE system~\eqref{eq:ODE:end}.
%Let $\Gamma_k^* \, x^{Y^*_s} = 0$ have $\ell$ classes.
Then $d = d_1 + \cdots + d_{\ell}$ if and only if  $L=L_1 \oplus \cdots \oplus L_{\ell}$.
\end{proposition}
\begin{proof}
Recall $d = |V^\sqcup_s| - \ell - \dim(L)$ and  $d_j = |V^j_s| - 1 - \dim(L_j)$.
Clearly, $V_s^\sqcup = V_s^1 \sqcup \cdots \sqcup V_s^\ell$ implies $|V_s^\sqcup| = |V_s^1| + \cdots + |V_s^\ell|$.
Hence, 
\[
d = |V^\sqcup_s| - \ell - \dim(L) = \sum_{j=1}^\ell d_j = \sum_{j=1}^\ell \left( |V^j_s| - 1 - \dim(L_j) \right)
\]
if and only if $\dim(L) =\sum_{j=1}^\ell \dim(L_j)$, that is, $L = L_1 \oplus \cdots \oplus L_\ell$. 
\end{proof}

\blue{For a summary of the notation introduced in this section, see Table~\ref{tab}(bc). For an
illustration, we continue the example from Sections~\ref{sec:previous} and~\ref{sec:reaction_networks}.}

\blue{
\begin{running}[continued]
For the last time, we return to the reaction network (with rate constants)
\begin{gather*}
    0 \xleftrightarrows{k_{12}}{k_{21}} X_1 , \quad
    2X_1 \xleftarrow[k_{43}]{} 3 X_1 \xrightarrow{k_{45}} 2X_1 + X_2 \xleftrightarrows{k_{56}}{k_{65}} X_1 + 2X_2 .
\end{gather*}
The decomposition of the kernel of the stoichiometric matrix $N = \begin{pmatrix} N^1 & N^2 \end{pmatrix}$ with 
\small
\[
N^1 =
\bordermatrix{& \gray{12} & \gray{21} & \gray{43} \cr
& 1 &-1 &-1 \cr
& 0 & 0 & 0 
}
\quad \text{and} \quad
N^2 =
\bordermatrix{& \gray{45} & \gray{56} & \gray{65} \cr
&-1 &-1 & 1 \cr
& 1 & 1 &-1
}
\]
\normalsize
as a direct product, that is, $\ker N = \ker N^1 \times \ker N^2$ implies a partition of the edge set,
$E = E^1 \dot\cup E^2$ with $E^1 = \{12,21,43\}$ and $E^2 = \{45,56,65\}$,
and hence a decomposition of the reaction network $(G,y)$ into $\ell=2$ independent subnetworks.
\begin{center}
\begin{tikzpicture}[scale=2]
    \draw [step=1, gray, very thin] (0,0) grid (3,2);
\draw [ ->, black!70!white] (0,0)--(3.25,0);
\draw [ ->, black!70!white] (0,0)--(0,2.25);

\begin{scope}[shift={(0,0)}]

\node[circle, fill=blue, inner sep=1pt, outer sep=5pt] (A) at (0,0) {};
\node[circle, fill=blue, inner sep=1pt, outer sep=5pt] (B) at (1,0) {};
\node[circle, fill=blue, inner sep=1pt, outer sep=5pt] (C) at (2,0) {};
\node[circle, fill=blue, inner sep=1pt, outer sep=5pt] (D) at (3,0) {};
\node[circle, fill=blue, inner sep=1pt, outer sep=5pt] (E) at (2,1) {};
\node[circle, fill=blue, inner sep=1pt, outer sep=5pt] (F) at (1,2) {};

\node[inner sep=0,outer sep=1] (0) at (0,0) {};

\draw[arrows={-stealth},very thick,red,transform canvas={yshift=2pt}] (A) to (B);
\draw[arrows={-stealth},very thick,red,transform canvas={yshift=-2pt}] (B) to (A);
\draw[arrows={-stealth},very thick,red,transform canvas={yshift=0pt}] (D) to (C);
\draw[arrows={-stealth},very thick,green,transform canvas={xshift=0pt}] (D) to (E);
\draw[arrows={-stealth},very thick,green,transform canvas={xshift=1.4pt,yshift=1.4pt}] (E) to (F);
\draw[arrows={-stealth},very thick,green,transform canvas={xshift=-1.4pt,yshift=-1.4pt}] (F) to (E);

\node [below left] (A) at (A) {$1$};
\node [below] (B) at (B) {$2$};
\node [below] (C) at (C) {$3$};
\node [below right] (D) at (D) {$4$};
\node [above right] (E) at (E) {$5$};
\node [above left] (F) at (F) {$6$};

%\node [below] (0) at (0) {$0$};

\end{scope}

% \begin{scope}[shift={(1.5,4)}]
%\node at (0,0) 

%{$\begin{aligned} 
%\frac{dx}{t} &=  k_1 xy + k_2x^2y^2  - k_3x^3y^2 + k_4x^3y^2 + k_5x^4y \\
%\frac{dy}{dt} &= k_1 xy - k_4 x^3y^2 
%\end{aligned}$};

\end{tikzpicture}
\end{center}
In the ``embedded'' graph,
the subnetworks $(G^1,y^1)$ and $(G^2,y^2)$ are shown in red and green, respectively.
Note that $G^1=(V^1,E^1)$ has vertex set $V^1 = \{ 1,2,3,4 \}$, whereas $V^2 = \{ 4,5,6 \}$ for $G^2=(V^2,E^2)$,
and hence $V^1 \cap V^2 = \{4\} \neq \emptyset$.
Accordingly, we introduce the disjoint union of vertex sets \[ V^\sqcup = V^1 \sqcup V^2 = \{ 1,2,3,4^1,4^2,5,6 \} , \]
which contains two ``copies'' of the element in the intersection, $4 \in V^1 \cap V^2$,
and the disjoint union of the source vertex sets $V_s^\sqcup = V_s^1 \sqcup V_s^2 = \{ 1,2,4^1,4^2,5,6 \}$.

From the matrices for the subnetworks, we define the corresponding block(-diagonal) combined matrices:
the combined $V^\sqcup \times E$ incidence and $V_s^\sqcup \times E$ source matrices,
\small
\[
I_E^* = 
\bordermatrix{& \gray{12} & \gray{21} & \gray{43} & \gray{45} & \gray{56} & \gray{65} \cr
\gray{1}   &-1 & 1 & 0 & 0 & 0 & 0 \cr
\gray{2}   & 1 &-1 & 0 & 0 & 0 & 0 \cr
\gray{3}   & 0 & 0 & 1 & 0 & 0 & 0 \cr
\gray{4^1} & 0 & 0 &-1 & 0 & 0 & 0 \cr
\gray{4^2} & 0 & 0 & 0 &-1 & 0 & 0 \cr
\gray{5}   & 0 & 0 & 0 & 1 &-1 & 1 \cr
\gray{6}   & 0 & 0 & 0 & 0 & 1 &-1
}
\quad \text{and} \quad
I_{E,s}^* = 
\bordermatrix{& \gray{12} & \gray{21} & \gray{43} & \gray{45} & \gray{56} & \gray{65} \cr
\gray{1}   & 1 & 0 & 0 & 0 & 0 & 0 \cr
\gray{2}   & 0 & 1 & 0 & 0 & 0 & 0 \cr
\gray{4^1} & 0 & 0 & 1 & 0 & 0 & 0 \cr
\gray{4^2} & 0 & 0 & 0 & 1 & 0 & 0 \cr
\gray{5}   & 0 & 0 & 0 & 0 & 1 & 0 \cr
\gray{6}   & 0 & 0 & 0 & 0 & 0 & 1
} ,
\]
\normalsize
the combined $V^\sqcup \times V_s^\sqcup$ (rectangular) Laplacian matrix,
\small
\[
R_k^* = I_E^* \diag(k) (I_{E,s}^*)^\trans = 
\bordermatrix{& \gray{1} & \gray{2} & \gray{4^1} & \gray{4^2} & \gray{5} & \gray{6} \cr
\gray{1}   &-k_{12} & k_{21} &      0 &      0 &      0 &      0 \cr
\gray{2}   & k_{12} &-k_{21} &      0 &      0 &      0 &      0 \cr
\gray{3}   &      0 &      0 & k_{43} &      0 &      0 &      0 \cr
\gray{4^1} &      0 &      0 &-k_{43} &      0 &      0 &      0 \cr
\gray{4^2} &      0 &      0 &      0 &-k_{45} &      0 &      0 \cr
\gray{5}   &      0 &      0 &      0 & k_{45} &-k_{56} & k_{65} \cr
\gray{6}   &      0 &      0 &      0 &      0 & k_{56} &-k_{65}
} ,
\]
\normalsize
the combined $n \times V^\sqcup$ complex and $n \times V_s^\sqcup$ source complex matrices,
\small
\[
Y^* = 
\bordermatrix{& \gray{1} & \gray{2} & \gray{3} & \gray{4^1} & \gray{4^2} & \gray{5} & \gray{6} \cr
& 0 & 1 & 2 & 3 & 3 & 2 & 1 \cr
& 0 & 0 & 0 & 0 & 0 & 1 & 2
}
\quad \text{and} \quad
Y_s^* = 
\bordermatrix{& \gray{1} & \gray{2} & \gray{4^1} & \gray{4^2} & \gray{5} & \gray{6} \cr
& 0 & 1 & 3 & 3 & 2 & 1 \cr
& 0 & 0 & 0 & 0 & 1 & 2
} ,
\]
\normalsize
and the combined $n \times V_s^\sqcup$ kinetic matrix,
\small
\[
\Gamma_k^* = Y^* R_k^* = 
\bordermatrix{& \gray{1} & \gray{2} & \gray{4^1} & \gray{4^2} & \gray{5} & \gray{6} \cr
& k_{12} &-k_{21} &-k_{43} & -k_{45} &-k_{56} & k_{65} \cr
& 0      & 0      & 0      & k_{45}  & k_{56} &-k_{65}
} .
\]
\normalsize
The latter defines the combined kinetic subspace, $K^* = \im \Gamma_k^* = \R^2$ (for all $k$).

\end{running}
}

% ========= ========= ========= ========= ========= ========= ========= =========

\begin{table}[htp]
\small
\blue{
\[
\begin{array}{@{}llll@{}}
\toprule
\text{symbol} & \text{dimension} & \text{name} & \text{relation} \\
\midrule  
\multicolumn{4}{c}{\text{\bf (a) reaction network}} \\    
I_E & V \times E & \text{incidence matrix} \\
I_{E,s} & V_s \times E & \text{source matrix} \\
R_k & V \times V_s & \text{(rectangular) Laplacian matrix} & R_k = I_E \diag(k) (I_{E,s})^\trans \\
Y & n \times V & \text{complex matrix} \\
Y_s & n \times V_s & \text{source complex matrix} \\
N & n \times E & \text{stoichiometric matrix} & N = Y I_E \\
\Gamma_k & n \times V_s & \text{kinetic matrix} & \Gamma_k = Y R_k\\
%\midrule
S & & \text{stoichiometric subspace} & S = \im N \\
K & & \text{kinetic subspace} & K = \im \Gamma_k \\
%\midrule
l & & \text{number of connected components} \\
t & & \text{number of absorbing strong comp.} \\
\delta & & \text{deficiency} & \delta = |V| - l - \dim(S) \\
\midrule
\multicolumn{4}{c}{\text{\bf (b) network decomposition}} \\    
\ell & & \text{number of (independent) subnetworks} \\  
\multicolumn{2}{@{}l}{I_E^j, I_{E,s}^j, R_k^j, Y^j, \ldots} & 
\text{objects in (a) for } j = 1, \ldots, \ell \\
V^\sqcup & & \text{disjoint union of vertex sets} & V^\sqcup = V^1 \sqcup \cdots \sqcup V^\ell \\
V_s^\sqcup & & \text{disjoint union of source vertex sets} & V_s^\sqcup = V_s^1 \sqcup \cdots \sqcup V_s^\ell \\
I_E^* & V^\sqcup \times E & \text{combined incidence matrix} \\
I_{E,s}^* & V_s^\sqcup \times E & \text{combined source matrix} \\
R_k^* & V^\sqcup \times V_s^\sqcup & \text{combined (rect.) Laplacian matrix} & R_k^* = I_E^* \diag(k) (I_{E,s}^*)^\trans \\
Y^* & n \times V^\sqcup & \text{combined complex matrix} \\
Y_s^* & n \times V_s^\sqcup & \text{combined source complex matrix} \\
\Gamma_k^* & n \times V_s^\sqcup & \text{combined kinetic matrix} & \Gamma_k^* = Y^* R_k^* \\ 
K^* & & \text{combined kinetic subspace} & K^* = \im \Gamma_k^* \\ 
\midrule
\multicolumn{4}{c}{\text{\bf (c) polynomial equations}} \\ 
P & & \text{coefficient polytope} \\
L & & \text{monomial difference subspace} \\
D & & \text{monomial dependency subspace} \\
d & & \text{monomial dependency} & d = \dim D \\
\multicolumn{2}{@{}l}{P_j, L_j, D_j, d_j} & \text{objects for } j = 1,\ldots,\ell \\
\bottomrule
\end{array}
\]
\captionsetup{font=small}
\caption{
\blue{Notation introduced in Sections~\ref{sec:reaction_networks} and \ref{sec:network_decomposition}.
(a) Matrices and subspaces 
for a reaction network $(G,y)$ / a mass-action system $(G_k,y)$
with underlying graph $G=(V,E)$ and edge labels~$k$.
(b) Network decomposition:
objects in (a) for subnetworks $(G^j,y^j)$ with $G^j=(V^j,E^j)$ and combined objects for the full network $(G,y)$.
(c) Polytope and subspaces for the parametrized system of polynomial equations $\Gamma_k^* \, x^{Y^*_s} = 0$,
arising from the decomposed dynamical system~\eqref{eq:ODE:end}.
See Definition~\ref{definition} for $A \left( c\circ x^B \right) = 0$ with $A=\Gamma_k^*$, $B=Y^*_s$, and $c=1$;
analogously for the subsystems $\Gamma_k^j \, x^{Y_s^j} = 0$.}
}
\label{tab}
}
\end{table}

% ========= ========= ========= ========= ========= ========= ========= =========

\blue{
\section{Dependency one mass-action systems} \label{sec:main}
}

We present a {\em dependency} one theorem for mass-action systems,
extending the deficiency one theorems by Feinberg~\cite{Feinberg1987,Feinberg1995a},
cf.\ Theorem~\ref{thm:def}.

In Theorem~\ref{thm:main:mass-action} below, 
we neither assume $\delta_i \le 1$, $t_i = 1$, nor $K=S$.
\blue{
Most importantly, we need to guarantee $d = \dim P$ (per class), in order to apply Theorem~\ref{thm:main}.
}
%In fact, for one class, $K$ and $L$ determine the desired equality of $d$ and $\dim P$.

\begin{lemma} \label{lem:d_dimP} 

\blue{
Let $(G,y)$ be a reaction network with one subnetwork (the network itself)
and let $(G_k,y)$ be the corresponding mass-action system.
Consider the equilibrium equation $\Gamma_k \, x^{Y_s} = 0$  (with one class) and assume $\ker \Gamma_k \cap \R^{V_s}_> \neq \emptyset$.}
%Let $\Gamma_k^* \, x^{Y_s} = 0$ have \emph{one class} and assume $\ker \Gamma_k^* \cap \R^{V_s}_> \neq \emptyset$. 
Then, 
\[
\dim P = d 
\quad \text{if and only if} \quad
\dim K = \dim L . 
\]
\end{lemma}
\begin{proof}
On the one hand, 
$\dim P + 1 = \dim C = \dim (\ker \Gamma_k) = \dim (\ker (Y R_k)) = |V_s| - \dim (\im (Y R_k)) = |V_s| - \dim K$.
On the other hand,
$d + 1 = |V_s| - \dim L$.
Hence, $\dim P = d$ if and only if $\dim K = \dim L$.
\end{proof}

\blue{
As stated above,
we do not explicitly restrict the number of \blue{absorbing} strong components (per class), in particular, if they are singletons.
However, we record that} $\dim P \le 1$ (an at most one-dimensional coefficient polytope) implies \blue{$t' \le 2$ (at most two non-singleton \blue{absorbing} strong components).}

\begin{fact} \label{fac:tprime} \marginnote{not used}
\blue{
Let $(G,y)$ be a reaction network with one subnetwork (the network itself)
and let $(G_k,y)$ be the corresponding mass-action system.
Consider the equilibrium equation $\Gamma_k \, x^{Y_s} = 0$ (with one class) and assume $\ker \Gamma_k \cap \R^{V_s}_> \neq \emptyset$.}
%Let $\Gamma_k^* \, x^{Y_s} = 0$ have \emph{one class} and assume $\ker \Gamma_k^* \cap \R^{V_s}_> \neq \emptyset$. 
Then, 
\[
\dim P \le 1 \implies \blue{t' \le 2 .}
\]
\blue{In fact, $t'=2$ only if there are two components ($l=2$) which are strongly connected.}
\end{fact}
\begin{proof}
Let $\dim P \le 1$. Then, $t' \le 2$ since
\begin{align*}
\dim P+1 &= \dim C = \dim (\ker \Gamma_k) = \dim (\ker (Y R_k)) \\
&= \dim (\ker Y \cap \im R_k) + \underbrace{\dim (\ker R_k)}_{t'} . 
\end{align*}
\blue{
If $t' = 2$, then $\ker \Gamma_k = \ker R_k$ is generated by two nonnegative vectors with support on the two non-singleton \blue{absorbing} strong components.
Since the union of their supports is the set of all source vertices,
there are two components, and they are strongly connected. 
}
\end{proof}

\begin{theorem}[$d_j\le1$, mass-action systems] \label{thm:main:mass-action}
Let $(G_k,y)$ be a mass-action system with $\ell$ independent subnetworks. 
\blue{Recall the kinetic and stoichiometric subspaces, $K$ and $S$, respectively.
For the equilibrium equation~$\Gamma^*_k \, x^{Y_s^*} = 0$,
recall the monomial dependency~$d$ and the monomial difference subspace~$L$.} 
Let $(G_k,y)$ fulfill the following conditions:
\begin{enumerate}
\item[(I)]
$\ker \Gamma_k^* \cap \R^{V^\sqcup_s}_> \neq \emptyset$.
\item[(II)] $d = d_1 + \cdots + d_{\ell}$.
\item[(III)]
For every (class) $j=1,\ldots,\ell$, 
\begin{itemize}
    \item 
    $d_j \le 1$ and $\dim(K_j)= \dim(L_j)$.
    \item 
    If $d_j=1$, then
    \begin{itemize}
        \item 
        $\displaystyle \sum_{i'=1}^{i} \bar b^j_{i'} \ge 0$ for all $i=1,2,\ldots,\omega_j-1$
        (or ``$\le0$'' for all $i$)
        \item
        $\bar b^j_1 \cdot \bar b^j_{\omega_j} < 0$.
    \end{itemize}
    (The components of the (lumped) dependency vector $\bar b^j \in \R^{\omega_j}$ are ordered with respect to the vector $\bar q^j \in \R^{\omega_j}$ of the polytope~$P_j$, \blue{see~Definition~\ref{definition_decomposable systems}.}) 
\end{itemize}
\item[(IVa)]
\blue{$L=K$} or (IVb) $L=S$.
\end{enumerate}
Then,
there exists a unique positive equilibrium within every (IVa) \underline{kinetic} or (IVb) \underline{stoichiometric} compatibility class.
\end{theorem}
\begin{proof}
The mass-action system $(G_k,y)$ gives rise to the polynomial equations $\Gamma_k^* \, x^{Y_s^*} = 0$ with $\ell$ classes.
In order to apply Theorem~\ref{thm:main} \blue{(for $A \left( c \circ x^B \right) = 0$ with $A=\Gamma_k^*$, $B=Y^*_s$, and $c=1$),}
\blue{we show that conditions (i), (ii), (iii) there
follow from conditions (I), (II), (III) here.}

(i) By (I) with $A = \Gamma_k^*$.

(ii) By (II). 

(iii) By (III), $\dim(K_j)= \dim(L_j)$, and by Lemma~\ref{lem:d_dimP}, $d_j = \dim P_j$ (for $j=1,\ldots,\ell$). 
Regarding $d_j=1$, (iii) and (III) are identical. 

By Theorem~\ref{thm:main}, 
the solution set on the coefficient polytope is a singleton, $|Y_c|=1$,
and by Theorem~\ref{thm:previous}, the solution set is an exponential fiber,
$Z_c = x^* \circ \e^{L^\perp}$ with $x^* \in \R^n_>$. 

The set of positive equilibria within the kinetic compatibility class given by $x' \in \R^n_>$ is $x^* \circ \e^{L^\perp} \cap\, (x' + K)$.
By (IVa), $K=L$, and by Theorem~\ref{thm:birch}, the
intersection is a singleton. 

The set of positive equilibria within the stoichiometric compatibility class given by $x' \in \R^n_>$ is $x^* \circ \e^{L^\perp} \cap \,(x' + S)$.
By (IVb), $L=S$, and by Theorem~\ref{thm:birch}, the
intersection is a singleton. 
\end{proof}

\blue{We now provide a simplified version of Theorem~\ref{thm:main:mass-action} that is closer to the (extended) deficiency one theorem
(Theorem~\ref{thm:def} in Section~\ref{subsec:def1thm}).
% In particular, 
% we consider the finest decomposition of $\ker N$ (and the resulting independent subnetworks),
% which implies a decomposition of $\ker \Gamma_k$,
% cf.~Proposition~\ref{pro:decomp}.
% That is, the classes are determined by the independent subnetworks.
In particular, we assume that the kinetic, stoichiometric, and monomial difference subspaces of the subnetworks (and hence of the full system) agree.
}

\blue{
\begin{corollary} \label{cor:main:mass-action}
Let $(G_k,y)$ be a mass-action system with $\ell$ independent subnetworks. 
\blue{Recall the kinetic and stoichiometric subspaces, $K$ and $S$, respectively.
For the equilibrium equation~$\Gamma^*_k \, x^{Y_s^*} = 0$,
recall the monomial dependency~$d$ and the monomial difference subspace~$L$.} 
Let $(G_k,y)$ fulfill the following conditions:
\begin{enumerate}
\item [(I')]
$\ker \Gamma_k^* \cap \R^{V^\sqcup_s}_> \neq \emptyset$.
\item [(III')]
For every (independent subnetwork) $j=1,\ldots,\ell$, 
\begin{itemize}
    \item 
    $d_j \le 1$ and $K_j = S_j = L_j$.
    \item 
    If $d_j=1$, then
    \begin{itemize}
        \item 
        $\displaystyle \sum_{i'=1}^{i} \bar b^j_{i'} \ge 0$ for all $i=1,2,\ldots,\omega_j-1$
        (or ``$\le0$'' for all $i$)
        \item
        $\bar b^j_1 \cdot \bar b^j_{\omega_j} < 0$.
    \end{itemize}
    (The components of the (lumped) dependency vector $\bar b^j \in \R^{\omega_j}$ are ordered with respect to the vector $\bar q^j \in \R^{\omega_j}$ of the polytope~$P_j$, \blue{see~Definition~\ref{definition_decomposable systems}.}) 
\end{itemize}
\end{enumerate}
Then,
there exists a unique positive equilibrium within every  \underline{stoichiometric} compatibility class.
\end{corollary}

\begin{proof}
The mass-action system $(G_k,y)$ gives rise to the polynomial equations $\Gamma_k^* \, x^{Y_s^*} = 0$ with $\ell$ classes (determined by the $\ell$ independent subnetworks), cf.~Proposition~\ref{pro:decomp}.
In order to apply Theorem~\ref{thm:main:mass-action},
we show that conditions (I), (II), (III), (IVb) there
follow from conditions (I'), (III') here.

(I) By (I').

(II) By (III'), $L_j=S_j$,
and by the decomposition of $\ker N$, $S=S_1 \oplus \cdots \oplus S_{\ell}$.
Hence, $L=L_1 \oplus \cdots \oplus L_{\ell}$,
and (II) by Proposition~\ref{pro:decomposable}.

(III) By (III').

(VIb) By (III'), see the argument for (II).
\end{proof}
}

% In order to apply Theorem~\ref{thm:main} for the polynomial equations $A \, (c \circ x^B) = 0$,
% we show that conditions (i), (ii), (iii) there
% follow from conditions (I), (II) here.

% (i) By (I) with $A = \Gamma_k^*$.

% (ii) By the decomposition of $\ker N$, $S=S_1 \oplus S_2 \oplus\cdots \oplus S_{\ell}$,
% and by assumption (II), $L_j=S_j$.
% Hence, $L=L_1 \oplus L_2 \oplus\cdots \oplus L_{\ell}$,
% and (ii) by Proposition~\ref{pro:decomposable}. 

% (iii) By (II), $\dim(K_j)= \dim(L_j)$, and by Lemma~\ref{lem:d_dimP}, $d_j = \dim P_j$ (for $j=1,\ldots,\ell$). 
% Regarding $d_j=1$, (iii) and (II) are identical. 

% By Theorem~\ref{thm:main}, 
% the solution set on the coefficient polytope is a singleton, $|Y_c|=1$,
% and by Theorem~\ref{thm:previous}, the solution set is an exponential fiber,
% $Z_c = x^* \circ \e^{L^\perp}$ with $x^* \in \R^n_>$. 

% The set of positive equilibria within the stoichiometric compatibility class given by $x' \in \R^n_>$ is $x^* \circ \e^{L^\perp} \cap (x' + S)$.
% By the argument above, $L=S$, and by Theorem~\ref{thm:birch}, the
% intersection is a singleton. 

% ========= ========= ========= ========= ========= ========= ========= =========

\blue{
\section{Deficiency one mass-action systems} \label{sec:deficiency_one}

In the setting of the deficiency one theorem, 
we assume that the network is decomposed into $\ell$ independent subnetworks with one (absorbing) strong component.
That is, the classes arise from $\ker N$ and they are assumed to be {\em connected}.
This implies $\delta = \delta_1 + \cdots + \delta_\ell$,
see Fact~\ref{fac:independent}.
The main result of this subsection is Lemma~\ref{lem:graphical},
a graph-theoretical argument used in the proof for the existence of a positive equilibrium
in the case of weakly reversible networks,
see Theorem~\ref{thm:def} in Subsection~\ref{subsec:def1thm} and its proof in Subsection~\ref{subsec:proof}.
}

\newcommand{\vertC}{\blue{\mathcal{C}}}
\newcommand{\vertV}{\blue{\mathcal{V}}}

Given the directed graph $G=(V,E)$ of a network, we introduce the bipartite graph $\mathcal{G} = (\vertC,\vertV,\mathcal{E})$ with the first vertex set $\vertC = \{1,\ldots,\ell\}$ representing subnetworks of $G$,
the second vertex set $\vertV \subseteq V$ representing ``shared'' vertices (appearing in more than one subnetwork),
and the edge set $\mathcal{E} = \{ (j,i) \in \vertC \times \vertV \mid i \in V^j \}$, connecting shared vertices with their subnetworks.

\begin{lemma} \label{lem:graphical}
Let $(G,y)$ \blue{with $G=(V,E)$} be a reaction network decomposed into $\ell$ connected, independent subnetworks.
\blue{Recall the (combined) incidence matrices, $I_E$ and $I_E^*$, the number $l$ of connected components, and the disjoint union $V^\sqcup$ of source vertex sets.}
Then, $|V^\sqcup| - |V| = \ell - l$,
that is, $\ker I_E = \ker I_E^*$.
\end{lemma}
\begin{proof}
Since we assume that subgraphs are connected,
all vertices of a given subgraph derive from only one component of the original graph,
and we can consider the decomposition of $l=1$ component into $\ell$ subgraphs. 
For the resulting bipartite graph $\mathcal{G}$,
$|\vertC| = \ell$ and 
\[
|\mathcal{E}| = \sum_{i \in \vertV} \text{deg}(i) = \sum_{i \in \vertV} (\text{deg}(i) - 1 + 1) = V^\sqcup - V + |\vertV| .
\]
Since $\mathcal{G}$ is connected, $|\mathcal{E}| \ge |\vertC| + |\vertV| -1$, that is, $V^\sqcup - V \ge \ell - 1$.
For clarity, note that a ``double edge'' $(i j i')$ indicates that the shared vertices $i, i'$ reside in subgraph~$j$
and hence $0 \neq y(i')-y(i) \in S_j$.
Now, assume $V^\sqcup - V > \ell - 1$.
Then, there is a cycle of $n\ge2$ ``double edges'' with vertex sequence $i_1 j_1 i_2 j_2 \ldots i_n j_n i_{n+1}$,
where $i_{n+1}=i_1$, $i_* \in \vertV \subseteq V$, $j_* \in \vertC = \{1,\ldots,\ell\}$. 
Hence,
\[
\underbrace{y(i_2) - y(i_1)}_{\in S_{j_1}} + \cdots + \underbrace{y(i_1) - y(i_n)}_{\in S_{j_n}} = 0 ,
\]
contradicting $S = S_1 \oplus \cdots \oplus S_\ell$.
Hence, $V^\sqcup - V = \ell - 1$.
By combining the arguments for the $l$ components of the original graph, 
$V^\sqcup - V = \ell - l$.

Recall $I_E \in \{-1,0,1\}^{V \times E}$ and $I_E^* \in \{-1,0,1\}^{V^\sqcup \times E}$.
Obviously, $\ker I_E \supseteq \ker I_E^*$.
By the rank-nullity theorem,
$\dim (\ker I_E) = |E| - \dim (\im I_E) = |E| - \dim (\im I_E^\trans) 
= |E| - |V| + \dim (\ker I_E^\trans) = |E| - |V| + l$
as well as
$\dim (\ker I_E^*) = |E| - |V^\sqcup| + \ell$.
Hence,
$\dim (\ker I_E) - \dim (\ker I_E^*) = |V^\sqcup| - |V| - (\ell-l)$. %$ = 0$
%and $\ker I_E = \ker I_E^*$.
\end{proof}

\begin{fact} \label{fac:independent} \marginnote{not used}
Let $(G,y)$ %\blue{with $G=(V,E)$} 
be a reaction network decomposed into $\ell$ connected subnetworks. 
%\blue{Recall the stoichiometric subspace $S$, the deficiency $\delta$, the number of connected components, $l$, the disjoint union of source vertex sets, $V^\sqcup$.}
The following statements are equivalent.
\begin{enumerate}
\item 
$S = S_1 \oplus \cdots \oplus S_\ell$.
\item
$\delta = \delta_1 + \cdots + \delta_\ell$
and
$|V^\sqcup| - |V| = \ell - l$.
\end{enumerate}
\end{fact}
\begin{proof}
$2 \implies 1$.
On the one hand, $\delta = |V|-l-\dim S$.
On the other hand, $\delta_j = |V^j| - 1 - \dim S_j$ 
and hence $\sum_j \delta_j = |V^\sqcup| - \ell - \sum_j \dim S_j$.
Altogether,
\[
\sum_j \dim S_j - \dim S = \delta - \sum_j \delta_j + |V^\sqcup| - |V| - (\ell - l) .
\]
$1 \implies 2$.
By Lemma~\ref{lem:graphical}, 
if $S = S_1 \oplus \cdots \oplus S_\ell$,
then $|V^\sqcup| - |V| = \ell - l$.
By the displayed equation above, also $\delta = \sum_j \delta_j$.
\end{proof}

\blue{
In the ``classical'' setting ($\ell = l)$,
the connected subnetworks are the components of the graph (the linkage classes), 
and hence they are independent exactly when their deficiencies add up to the network deficiency.
}

\blue{
\subsection{The (extended) deficiency one theorem} \label{subsec:def1thm}
}

We prove Theorem~\ref{thm:def} below
which already extends the ``classical'' deficiency one theorem~\cite[Theorem~6.2.1]{Feinberg1987}
from components (``linkage classes'') to independent subnetworks, cf.\ \cite[Remark~6.2.D]{Feinberg1987} and \cite[Theorem~A.1]{Feinberg1995a}.

\begin{theorem}[$\delta_j \le 1$, independent subnetworks] \label{thm:def}
Let $(G_k,y)$ be a mass-action system with $\ell$ independent subnetworks.
\blue{Recall the deficiency $\delta$ and the number $t$ of absorbing strong components.}
Let $(G_k,y)$ fulfill the following conditions:

For every (independent subnetwork) $j=1,\ldots,\ell$, 
\begin{enumerate}[(i)]
\item 
$\delta_j\le 1$ and 
$t_j=1$.
\end{enumerate}
If (ii) there exists a positive equilibrium, % (in some stoichiometric compatibility 
then there exists a unique positive equilibrium within every stoichiometric compatibility class.

If the system is weakly reversible, then (ii) for all rate constants $k$.
\end{theorem}

\begin{remark}[\blue{on} the statement]
Note that in the deficiency one theorems, the existence of a positive equilibrium is assumed. In our dependency one theorem, we only assume the necessary condition $\ker \Gamma_k^* \cap \R^{V^\sqcup_s}_> \neq \emptyset$. 
Indeed, \blue{in Theorem~\ref{thm:def}, the existence assumption can be replaced by the necessary condition.}
A different (more involved) condition was given by Boros~\cite{boros2013existence}.
\end{remark}

\begin{remark}[\blue{on} the proof]
\blue{A detailed} proof of the classical theorem \blue{is given} in \cite[Sections~5--8 \blue{and} Appendix~B]{Feinberg1995a}.
For an outline of the proof of the extension (to independent subnetworks or ``direct partitions''), see \cite[Appendix~A]{Feinberg1995a}. 
In~\cite[Sections~5 \blue{and} 6]{Feinberg1995a},
Feinberg \blue{shows} 
that the set of positive equilibria has the form $x^* \circ e^{S^\perp}$ (a particular solution $x^*$ \blue{times} the exponentiation of the orthogonal complement of the stoichiometric subspace $S$),
by \blue{reformulating} the problem using a monotonic function.
\blue{Our approach is based on} a geometric framework for parametrized systems of polynomial equations and inequalities~\cite{MuellerRegensburger2023b,MuellerRegensburger2023a}. 
\blue{Within this framework,}
we establish the existence of a point $y^*$ on the coefficient polytope
(\blue{where $y^*$ is in one-to-one correspondence with} $x^* \circ e^{S^\perp}$).
As Feinberg,
we use the ``second salt theorem'' (a network argument) and \blue{a version of} ``Birch's theorem''.
In the case of weakly reversible networks,
we replace the proof for the existence of a positive equilibrium in \cite[Section~8]{Feinberg1995a}
by a graph-theoretical argument
\blue{(see the last paragraph in the proof of Theorem~\ref{thm:def}, which uses Lemma~\ref{lem:graphical}).}
\end{remark}

\blue{In Subsection~\ref{subsec:proof},}
we will show that the deficiency one theorem follows from our more general dependency one results,
\blue{Theorem~\ref{thm:main:mass-action} or its Corollary~\ref{cor:main:mass-action}.}
Indeed, we exhibit how the network conditions in Theorem~\ref{thm:def}
ensure the conditions in \blue{Corollary~\ref{cor:main:mass-action}.}

%First, connectedness of the independent subnetworks implies $L \subseteq S$ (for the subnetworks). 

% \begin{lemma} 
% Let $(G_k,y)$ be a mass-action system
% with $l=1$ (one component).
% Then, $L \subseteq S$.
% \end{lemma}
% %
% \begin{proof}
% Let $I_\mathcal{E} \in \{-1,0,1\}^{V \times \mathcal{E}}$
% be the incidence matrix of an auxiliary graph $(V,\mathcal{E})$ with ${|\mathcal{E}| = |V| - 1}$.
% %
% Since the graphs are connected---$(V,E)$ by assumption and $(V,\mathcal{E})$ by definition, $\im I_\mathcal{E} = \im I_E$ and hence
% \[
% L = \im (Y_s I_{\mathcal{E}_s}) \subseteq \im (Y I_\mathcal{E}) 
% = \im (Y I_E) = S .
% \]
% \end{proof}

\blue{
First,
we show that the kinetic, stoichiometric, and monomial difference subspaces of the subnetworks agree.}

\begin{lemma} \label{lem:one_class} 
Let $(G_k,y)$ be a mass-action system with \blue{$t=1$ (one absorbing strong component).}
% and hence one component),
\blue{Recall the number $t'$ of non-singleton absorbing strong components
and the deficiency $\delta$. 
Further, recall the kinetic and stoichiometric subspaces, $K$ and $S$, respectively. 
For the equilibrium equation~$\Gamma_k \, x^{Y_s} = 0$, recall the dependency $d$ 
and the monomial difference subspace~$L$,
and assume} 
that $\ker \Gamma_k \cap \R^{V_s}_> \neq \emptyset$. 
Then, %\marginnote{do we need the kernel condition ?}
\begin{enumerate}[1.]
\item %\label{part_1}
$K=L=S$.
\item %\label{part_2} 
$d = \delta + t '-1$.
\end{enumerate}
\end{lemma}

\blue{
\begin{proof}
1.
Since $t=l=1$, $\im R_k = \im I_E$
and hence $K = \im(Y R_{k}) = \im(Y I_{E}) = S$.

If $t'=1$ (the absorbing strong component is not a singleton), then $V_s = V$ and $L=S$.
If $t'=0$ (the absorbing strong component is a singleton),
then $V = V_s \dot\cup \{i_*\}$ and $E = E_s \dot\cup E_*$,
where $E_s$ is the set of reactions between source vertices (and hence $E_*$ is the set of reactions with target vertex $i_*$).

Now, note that $S$ is the linear subspace associated with the affine hull of the complexes, $\aff(Y) = \aff (y(i) \mid i \in V)$,
and $L$ is the linear subspace associated with the affine hull of the source complexes, $\aff(Y_s) = \aff (y(i) \mid i \in V_s)$. We will show $y(i_*) \in \aff(Y_s)$ which implies $\aff(Y_s) = \aff(Y)$ and hence $L=S$.

Since $\ker \Gamma_k \cap \R^{V_s}_> \neq \emptyset$, there is $\xi \in \R^{V_s}_>$ such that $\Gamma_k \, \xi = 0$, and hence there is $\alpha = \diag(k) (I_{E,s})^\trans \xi \in \R^E_>$ such that $Y I_E \, \alpha = 0$. 
That is, $Y I_{E_s} \alpha_{E_s} + Y I_{E_*} \alpha_{E_*} = 0$,
using $I_E = \begin{pmatrix} I_{E_s} & I_{E_*} \end{pmatrix}$ and $\alpha = \binom{\alpha_{E_s}}{\alpha_{E_*}}$.
Clearly, $Y I_{E_*} \alpha_{E_*} = - Y I_{E_s} \alpha_{E_s} \in L$.
Explicitly,
\[
\delta y := Y I_{E_*} \alpha_{E_*} 
= \sum_{(i \to i_*) \in E_*} \left( y(i_*) - y(i) \right) \alpha_{i \to i_*} 
= \bar \alpha \Bigg( y(i_*) - \underbrace{\sum_{(i \to i_*) \in E_*} \beta_{i \to i_*} \, y(i)}_{y} \Bigg) ,
\]
where  $\bar \alpha = \sum_{(i \to i_*) \in E_*} \alpha_{i \to i_*}$,
$\beta_{i \to i_*} = \alpha_{i \to i_*} / \bar \alpha$, and $\sum_{(i \to i_*) \in E_*} \beta_{i \to i_*} = 1$.
That is, $y(i_*) = y + \delta y / \bar \alpha$,
where $y \in \aff(Y_s)$ and $\delta y \in L$.
Hence, also $y(i_*) \in \aff(Y_s)$. % which implies $L=S$.

2. Recall $d = |V_s| - 1 - \dim L$ and $\delta = |V| - 1 - \dim S$. 
By statement 1, $L=S$ and hence $d -\delta = |V_s| - |V| = t -t' = 1 - t'$. 
Hence $d = \delta + t'-1$.
\end{proof}

Regarding dimensions,
the assumption $t=1$ has two consequences: 
if $t'=1$ (non-singleton absorbing strong component), then $d=\delta$;
however, if $t'=0$ (singleton absorbing strong component), then $d=\delta-1$.
In particular,
$\delta=1$ and $t'=0$ imply $d=0$.

Second, to show that the sign conditions (for the dependency vector) in Corollary~\ref{cor:main:mass-action} are fulfilled,
we provide a variant of Feinberg's ``second salt theorem'' \cite{Feinberg2019}.
}

% ========= ========= ========= ========= ========= ========= ========= =========

\blue{
\subsection{The second salt theorem} \label{subsec:salt}
}

We continue our study of \blue{subnetworks} (with $\delta \le 1$, $t=1$, and $\ker (Y R_k) \cap \R^{V_s}_> \neq \emptyset$).
By Lemma~\ref{lem:d_dimP} and Lemma~\ref{lem:one_class}, $d=\dim P\le1$.
We consider the nontrivial case $d=\dim P=1$ which has $\delta=t'=1$,
again by Lemma~\ref{lem:one_class}.

Let $y^1, y^2 \in \ker (Y R_k) \cap \R^{V_s}_\ge$ be the vertices of $P$.
Since $t'=1$,
there is a nonnegative $\hat y \in \overline P$ with $R_k \, \hat y = 0$ and support on the (non-singleton) \blue{absorbing} strong component.
(It lies in the interior of the coefficient polytope,
if the network is weakly reversible,
and on the boundary, otherwise.)
Since $\delta = 1$,
there is a positive $\bar y \in P$ with $R_k \, \bar y = \beta \in \R^V$ and $\im \beta = \ker Y \cap \im R_k$.
(It lies in the interior,
possibly after adding $\hat y$.)
That is,
\begin{equation} \label{eq:two_ys}
\begin{aligned}
R_k \, \hat y &= 0 , \\
R_k \, \bar y &= \beta.
\end{aligned}
\end{equation}

Without loss of generality, we assume that the vertices $y^1$, $y^2$ of the coefficient polytope (a line segment) and the vectors $\hat y$, $\bar y$ are ordered as $y^1 \cdots \hat y \cdots \bar y \cdots y^2$, 
where $\hat y=y^1$ if the network is not weakly reversible.
Now, 
$\hat y = \hat \la y^1 + (1- \hat \la) y^2$ for some $\hat \la \in (0,1]$
and
$\bar y = \bar \la y^1 + (1- \bar \la) y^2$ for some $\bar \la \in (0,1)$,
where $\hat \la > \bar \la$.
Let
\[
\hat q = \hat y \circ (\bar y)^{-1} =
\left( \hat \la y^1 + (1- \hat \la) y^2 \right) \circ \left( \bar \la y^1 + (1- \bar \la) y^2 \right)^{-1} .
\]
It is easy to see%
\footnote{
For two vectors $y^1, y^2 \in \R^m_\ge$ with $y^1+ y^2, \, \bar \la y^1 + (1- \bar \la) y^2 \in \R^m_>$,
consider the maximal minors of the matrices $(y^1 \quad y^2), \, (y^1-y^2 \quad y^1+y^2), \, (\hat y \quad \bar y) \in \R^{m \times 2}$.
Explicitly, let
\begin{gather*}
(y^1 \quad y^2) = 
\begin{pmatrix}
a_1 & a_2 \\
b_1 & b_2 \\
\vdots & \vdots
\end{pmatrix}
, \qquad
(y^1-y^2 \quad y^1+y^2) = 
\begin{pmatrix}
a_1-a_2 & a_1+a_2 \\
b_1-b_2 & b_1+b_2 \\
\vdots & \vdots
\end{pmatrix}
, \\
(\hat y \quad \bar y) = 
\begin{pmatrix}
\hat \la a_1 + (1-\hat \la)a_2 & \bar \la a_1 + (1-\bar \la)a_2 \\
\hat \la b_1 + (1-\hat \la)b_2 & \bar \la b_1 + (1-\bar \la)b_2 \\
\vdots & \vdots
\end{pmatrix} .
\end{gather*}
If $\hat \la > \bar \la$, then
\begin{gather*}
\sign \left(
\frac{\hat \la a_1 + (1-\hat \la)a_2}{\bar \la a_1 + (1-\bar \la)a_2} - \frac{\hat \la b_1 + (1-\hat \la)b_2}{\bar \la b_1 + (1-\bar \la)b_2} 
\right) 
=
\sign \left(
\frac{a_1-a_2}{a_1+a_2} - \frac{b_1-b_2}{b_1+b_2}
\right) % \\
=
\sign
\begin{vmatrix}
a_1 & a_2 \\
b_1 & b_2\end{vmatrix} .
\end{gather*}
}
that the order of the entries of $\hat q$
agrees with the order of the entries of
\[
q = \left(y^1-y^2\right) \circ \left(y^1+y^2\right)^{-1}
\]
specified in Definition~\ref{definition_one_class}.

Finally, we introduce the scaled rectangular graph Laplacian 
\[
R_k \diag(\bar y) = R_{\bar k}
\quad \text{with} \quad \bar k = k \circ \bar y ,
\]
and we rewrite Equations~\eqref{eq:two_ys} as
\begin{equation} \label{eq:two_ys_scaled}
\begin{aligned}
R_{\bar k} \, \hat q &= 0 , \\
R_{\bar k} \, \ones_{V_s} &= \beta.
\end{aligned}
\end{equation}

Now, we state Feinberg's ``second salt theorem'' \cite{Feinberg2019} for the rectangular (rather than the square) graph Laplacian.
%This is the only argument that we use from Feinberg's proof~\cite{Feinberg1995a} (aside from ``Birch's theorem'' which is also used to prove the deficiency zero theorem).

\begin{lemma} \label{lem:salt}
Let $G_k=(V,E,k)$ be a labeled simple digraph with one component.
Let $T = \{1,\ldots,t\} \subseteq V = \{1,\ldots,m\}$ be (the vertices of) an \blue{absorbing} strong component, 
let $\hat q \in \R^{V_s}_\ge$ be the corresponding vector in the kernel of the rectangular graph Laplacian,
that is, $R_k \, \hat q = 0$ and $\supp \hat q = T$,
and assume $\hat q_1 \ge \cdots \ge \hat q_t > \hat q_{t+1} = \cdots = \hat q_m = 0$. % (after reordering of columns). 
Further, let $\beta = R_k \, \ones_{V_s} \in \R^V$.
Then, 
\[
\sum_{i'=1}^i \beta_{i'} \ge 0 ,
\quad \text{for } i \in T .
\]
If $i < t$ and $\hat q_{i} > \hat q_{i+1}$, then $\sum_{i'=1}^i \beta_{i'} > 0$. 
Finally,
$\sum_{i'=1}^t \beta_{i'} = 0$ if and only if $T=V$.
\end{lemma}
\begin{proof}
See Appendix~\ref{app:salt}. 
\end{proof}

% ========= ========= ========= ========= ========= ========= ========= =========

\blue{
\subsection{\texorpdfstring{Proof of the extended deficiency one theorem \\ (for independent subnetworks)}{}} \label{subsec:proof}
}

\begin{proof}[Proof of Theorem~\ref{thm:def}]
The mass-action system $(G_k,y)$ gives rise to the polynomial equations $\Gamma_k^* \, x^{Y_s^*} = 0$ with $\ell$ classes (determined by the $\ell$ connected, independent subnetworks), cf.~Proposition~\ref{pro:decomp}.
In order to apply Corollary~\ref{cor:main:mass-action},
we show that conditions (I'), (III') there follow from conditions (i), (ii) here. 

\begin{itemize}
\item[(I')] 
If (ii), there exists a positive equilibrium,
$x \in \R^n_>$ with $\Gamma_k^* \, x^{Y_s^*} = 0$,
and hence (I').
%implies $\ker \Gamma_k^* \cap \R^{V^\sqcup_s}_> \neq \emptyset$.

This further implies $\ker \Gamma_k^j \cap \R^{V^j_s}_> \neq \emptyset$
for the mass-action systems $(G^j_k,y^k)$, $j=1,\ldots,\ell$,
having $t_j=1$ by (i).
Hence, Lemma~\ref{lem:one_class} applies to the subsystems.

\item[(III')]
By Lemma~\ref{lem:one_class}, \blue{$K_j = S_j = L_j$ and $d_j \le 1$,}
$j=1,\ldots,\ell$. 
It remains to consider $d_j=1$ \blue{in detail, implying} $\delta_j=1$ and $t'_j=1$;
in particular, all vertices are source vertices, 
that is, $V^j_s = V^j$.

\blue{By the argument before Lemma~\ref{lem:salt},}
there are $\beta \in \R^{V^j}$ with $\im \beta = \ker Y^j \cap \im R^j_k$
and $\hat q \in \R^{V^j_s}_\ge$ \blue{such that} $R^j_{\bar k} \, \hat q=0$ and $R^j_{\bar k} \, 1=\beta$ (for a certain $\bar k \in \R^{V^j_s}_>$),
cf.~Equation~\eqref{eq:two_ys_scaled}.
In particular, $\supp \hat q = T$,
where $T \subseteq V^j$ denotes (the vertices of) the \blue{absorbing} strong component.
Let $t=|T|$ and $m=|V^j|$ and assume $\hat q_1 \ge \cdots \ge \hat q_t > \hat q_{t+1} = \cdots = \hat q_m = 0$. % 
By Lemma~\ref{lem:salt},
$
\sum_{i'=1}^i \beta_{i'} \ge 0 ,
\text{ for } i = 1,\ldots,t .
$
Further, if $i < t$ and $\hat q_{i} > \hat q_{i+1}$, then $\sum_{i'=1}^i \beta_{i'} > 0$.
Finally, $\sum_{i'=1}^t \beta_{i'} = 0$ if and only if $T=V^j$.

The polynomial system $\Gamma_k^j \, x^{Y^j_s}=0$ \blue{is of the form} ${A \, (c \circ x^B)=0}$ with $B = Y^j_s = Y^j$,
\blue{and} the dependency vector $b \in \R^{V^j}$ in Corollary~\ref{cor:main:mass-action}
is given by $\im b = D_k = \ker \binom{B}{\ones^\trans} = \ker \binom{Y^j}{\ones^\trans}$.
Since $\im R^j_k = \im I_E^j = \ker \ones^\trans$,
equivalently $\im b = \ker Y^j \cap \im R^j_k$.
That is, we can choose $b=\beta$ and the (in)equalities for $\beta$ guaranteed by Lemma~\ref{lem:salt} also hold for $b$.

Now, recall that $I_1, \ldots, I_\omega \subset V^j$ denote $\omega$ equivalence classes corresponding to equal (consecutive) components of~$\hat q$
and that $\bar b \in \R^\omega$ with $\bar b_i = \sum_{i' \in I_i} b_{i'}$
is the vector of lumped~$b$'s.
Since $\sum_{i'=1}^i b_{i'} \ge 0$ for $i=1,\ldots,t$, 
\blue{we get} ${\sum_{i'=1}^i \bar b_{i'} \blue{>} 0}$ for $i=1,2,\ldots,\omega-1$. 
\blue{(If $T \subset V^j$ and hence $I_\omega = \{t+1,\ldots,m\}$,
then $\sum_{i'=1}^{\omega-1} \bar b_{i'} = \sum_{i'=1}^t b_{i'} > 0$.)}
%Further, since $\hat q_i > \hat q_{i'}$ for $i \in I_1$ and $i' \in I_2$, 
%we get $\bar b_1 > 0$. 
Finally, since $\sum_{i=1}^\omega \bar b_i = 0$, we get $\bar b_\omega < 0$
and hence $\bar b_1\cdot \bar b_\omega <0$.
\end{itemize}

By Corollary~\ref{cor:main:mass-action}, there exists a unique positive equilibrium within every stoichiometric compatibility class.

\blue{It remains to show that weak reversibility implies (ii).}
%that is, the existence of a positive equilibrium for all $k$.
By Lemma~\ref{lem:graphical}, $\ker I_E = \ker I_E^*$,
and hence the cycles of the graph $G$ agree with the cycles of the subgraphs $G^j$, $j=1,\ldots,\ell$.
Hence, if $G$ is weakly reversible, then every $G^j$ is weakly reversible and $\ker R^j_k \cap \R^{V_s}_> \neq \emptyset$.
Since $\Gamma_k^j = Y^j R^j_k$, also $\ker \Gamma_k^j \cap \R^{V_s}_> \neq \emptyset$,
and since $\Gamma_k^* = \begin{pmatrix} \Gamma^1_k & \ldots & \Gamma^\ell_k \end{pmatrix}$, also $\ker \Gamma_k^* \cap \R^{V^\sqcup_s}_> \neq \emptyset$.
That is, (ii).
\end{proof}

% ========= ========= =========

%\clearpage
\section{Examples} \label{sec:examples}

\begin{example}
{\bf Deficiency two} network with one (singleton) \blue{absorbing} strong component.

\begin{figure}[!ht]
\begin{center}
\begin{tikzpicture}[scale=2]
    \input{fig_1_new}
\end{tikzpicture}
\caption{
\blue{Reaction network (as an embedded graph)} with 5 vertices and 4 source vertices
and hence $\delta = 5-1-2 = 2$, but $d = 4-1-2 = 1$.}
\end{center}
\end{figure}

The ODE associated with the mass-action system is given by
\[
\begin{split}
\dd{x_1}{t} &= k_{14} + k_{23} \, x_1^2 x_2 - k_{35} \, x_1^3 + k_{42} \, x_1 x_2 , \\
\dd{x_2}{t} &= k_{14} - k_{23} \, x_1^2 x_2 ,
\end{split}
\]
and steady state can be written as $\Gamma_k \, x^{Y_s}=0$ 
with
$\Gamma_k = \begin{pmatrix} k_{14} & k_{23}  & -k_{35} & k_{42} \\ k_{14} &  -k_{23} & 0  & 0
\end{pmatrix}$ 
and 
$Y_s = \begin{pmatrix} 0 & 2 & 3 & 1 \\ 0 & 1 & 0 & 1
\end{pmatrix}$.
We show that conditions (I), (II), (III), and (IVb) of Theorem~\ref{thm:main:mass-action} are satisfied.

(I) $\ker \Gamma_k\cap\R^4_{>0}$ is generated by 
$y^1 = (
\frac{1}{k_{14}} , \frac{1}{k_{23}} , \frac{2}{k_{35}} , 0
)^\trans$ 
and 
$y^2 = (
0 , 0 , \frac{1}{k_{35}} , \frac{1}{k_{42}} 
)^\trans$
and hence $\ker \Gamma_k\cap\R^4_{>0}\neq\emptyset$.
As a consequence, $q = ( 1 , 1 , q_3 , -1 )^\trans$ with $q_3 \in (-1,1)$.
(The vertices of the graph $G$ have been labeled such that the components of $q$ are ordered.)

(II) Satisfied trivially.

(III) 
Clearly, $d = 4-1-2 = 1$. Since $\dim(K) = \dim(\im{\Gamma_k})=2$, we get $K=L$. 
Finally, the (lumped) dependency vectors are given by $b=(
1,3,-1,-3
)^\trans$ 
and 
$\bar b = (
4,-1,-3
)^\trans$
and hence
$\displaystyle \sum_{i'=1}^{i} \bar b_{i'} \ge 0$ for $i=1,2$ 
as well as
$\bar b_1\cdot\bar b_3 = {4\cdot (-3) < 0}$.

(IVb) Since $L = \R^2$ holds for the monomial difference subspace (and $L \subseteq S$), we get $L=S$.

Using Theorem~\ref{thm:main:mass-action}, we find that there exists a unique positive equilibrium in every stoichiometric compatibility class, for all rate constants.
\end{example}

%\medskip
\clearpage

\begin{example} \label{exa2}
{\bf Deficiency two} network with {\bf \blue{two absorbing} strong components}
(one singleton and one non-singleton).

\begin{figure}[!ht]
\begin{center}
\begin{tikzpicture}[scale=2]
    \input{fig_2_new}
\end{tikzpicture}
\caption{
\blue{Reaction network (as an embedded graph)} with 5 vertices and 4 source vertices
and hence $\delta = 5-1-2 = 2$, but $d = 4-1-2 = 1$.}
\end{center}
\end{figure}
The ODE associated with the mass-action system is given by
\begin{alignat*}{5}
\dd{x_1}{t} &= - k_{15} \, x_1 - k_{23} \, x_1x_2 + 3k_{34} \, x_2 &&- 3k_{43} \, x_1^3 , \\
\dd{x_2}{t} &= + k_{12} \, x_1 - k_{21} \, x_1x_2 - k_{34} \, x_2  &&+ k_{43} \, x_1^3 ,
\end{alignat*}
and steady state can be written as $\Gamma_k \, x^{Y_s}=0$
with 
$\Gamma_k = \begin{pmatrix} -k_{15} & -k_{23}  & 3k_{34} & -3k_{43}  \\k_{12} &  - k_{21} & -k_{34} & k_{43}
\end{pmatrix}$ 
and 
$Y_s = \begin{pmatrix} 1 & 1 & 0 & 3  \\ 0 & 1 & 1  & 0
\end{pmatrix}$.
We show that conditions (I), (II), (III), and (IVb) of Theorem~\ref{thm:main:mass-action} are satisfied.

(I) $\ker \Gamma_k$ is generated by 
$y^1 = ( 3k_{21} + k_{23} , 3k_{12} - k_{15} , \frac{k_{12}k_{23} + k_{15}k_{21}}{k_{34}} , 0 )^\trans$ 
and 
$y^2 = ( 0 , 0 , \frac{1}{k_{34}} , \frac{1}{k_{43}} )^\trans$
and hence
$\ker \Gamma_k\cap\R^4_{>0}\neq\emptyset$
iff $3k_{12} - k_{15}>0$.
In this case,
$q = ( 1 , 1 , q_3 , -1 )^\trans$
with $q_3 \in (-1,1)$.
(The vertices of the graph $G$ have been labeled such that the components of $q$ are ordered.)

(II) Satisfied trivially.

(III)  
Clearly, $d = 4-1-2 = 1$. 
Since $\dim(K) = \dim(\im{\Gamma_k})=2$, we get $K=L$. 
Finally, the (lumped) dependency vectors are given by $b=(1,2,-2,-1)^\trans$ and $\bar b = (
3,-2,-1
)^\trans$
and hence
$\displaystyle \sum_{i'=1}^{i} \bar b_{i'} \ge 0$ for $i=1,2$ 
as well as
$\bar b_1\cdot\bar b_3 = {3\cdot (-1) < 0}$.

(IVb) Since $L = \R^2$ holds for the monomial difference subspace (and $L \subseteq S$), we get $L=S$.

Using Theorem~\ref{thm:main:mass-action}, we find that there exists a unique positive equilibrium in every stoichiometric compatibility class 
iff $3k_{12} - k_{15}>0$.
\end{example}

%\medskip
\clearpage

\begin{example}
{\bf Deficiency two} network with {\bf two} (singleton) {\bf \blue{absorbing} strong components}.

\begin{figure}[!ht]
\begin{center}
\begin{tikzpicture}[scale=2]
    \input{fig_3_new}
\end{tikzpicture}
\caption{
\blue{Reaction network (as an embedded graph)} with 5 vertices and 3 source vertices
and hence $\delta = 5-1-2 = 2$, but $d = 3-1-2 = 0$.}
\end{center}
\end{figure}
The ODE associated with the mass-action system is given by
\begin{alignat*}{11}
\dd{x_1}{t} &= (-k_{21} + k_{23}) \, x_1  &&+ ( -k_{32} + k_{34}) \, x_1^2  &&+ k_{45} \, x_1^3x_2 , \\
\dd{x_2}{t} &= k_{21} \, x_1 &&+ k_{34} \, x_1^2 &&- k_{45} \, x_1^3x_2 ,
\end{alignat*}
and steady state can be written as $\Gamma_k \, x^{Y_s}=0$
with
$\Gamma_k = \begin{pmatrix} -k_{21} + k_{23} & -k_{32} + k_{34} &  k_{45}  \\k_{21} &   k_{34} & -k_{45} 
\end{pmatrix}$ 
and 
$Y_s = \begin{pmatrix} 1 & 2 & 3  \\ 0 & 0 & 1 \end{pmatrix}$.
We show that conditions (I), (II), (III), and (IVb) of Theorem~\ref{thm:main:mass-action} are satisfied.

(I) $\ker \Gamma_k$ is generated by the vector $(\frac{k_{32} - 2k_{34}}{k_{21}k_{32} - 2k_{21}k_{34} + k_{23}k_{34}},\frac{k_{23}}{k_{21}k_{32} - 2k_{21}k_{34} + k_{23}k_{34}},\frac{1}{k_{45}})^\trans$
and hence
$\ker \Gamma_k\cap\R^4_{>0}\neq\emptyset$ iff $k_{32} - 2k_{34} >0$. 

(II) Satisfied trivially.

(III)  Clearly, $d = 3-1-2 = 0$. Since $\dim(K) = \dim(\im{\Gamma_k})=2$, we get $K=L$. 

(IVb) Since $L = \R^2$ holds for the monomial difference subspace (and $L \subseteq S$), we get $L=S$.

Using Theorem~\ref{thm:main:mass-action}, we find that there exists a unique positive equilibrium in every stoichiometric compatibility class
iff $k_{32} - 2k_{34} >0$. 
\end{example}

\subsection*{Acknowledgements}

This research was funded in whole, or in part, by the Austrian Science Fund (FWF), 
grant DOI 10.55776/PAT3748324 to SM.

For open access purposes, the authors have applied a CC BY public
copyright license to any author-accepted manuscript version arising
from this submission.

% \subsection*{Conflict of interest and data availability}

% On behalf of all authors, the corresponding author states 
% that there is no conflict of interest
% and that the manuscript has no associated data.

% ========= ========= ========= ========= ========= ========= ========= =========

%\clearpage

\addcontentsline{toc}{section}{References}
\bibliographystyle{abbrv}
\bibliography{old,new,MR}

@article{Boros2019,
  author = {Boros, Bal\'{a}zs},
  title = {Existence of positive steady states for weakly reversible mass-action systems},
  journal = {SIAM J. Math. Anal.},
  fjournal = {SIAM Journal on Mathematical Analysis},
  year = {2019},
  volume = {51},
  number = {1},
  pages = {435--449},
  doi = {10.1137/17M115534X}
}

@article{DeshpandeMueller2024,
    author = {Deshpande, Abhishek and M{\"u}ller, Stefan},
    title = {Existence of a unique, nondegenerate solution to parametrized systems of generalized polynomial equations},
    year = {2024},
    journal = {arXiv},
    note = {\href{https://arxiv.org/abs/2409.11288}{arXiv:2409.11288} [math.AG]},
}

@article{MuellerRegensburger2023a,
    author = {M{\"u}ller, Stefan and Regensburger, Georg},
    title = {Parametrized systems of generalized polynomial
inequalities via linear algebra and convex geometry},
    year = {2026},
    journal = {Positivity},
    volume = {30},
    pages = {4},
    doi = {10.1007/s11117-025-01158-4}
}

@article{MuellerRegensburger2023b,
    author = {M{\"u}ller, Stefan and Regensburger, Georg},
    title = {\phantom{}{P}arametrized systems of generalized polynomial equations: applications to fewnomials},
    year = {2023},
    journal = {arXiv},
    note = {\href{https://arxiv.org/abs/2304.05273}{arXiv:2304.05273} [math.AG]},
}

@article{MuellerRegensburger2023,
    author = {M{\"u}ller, Stefan and Regensburger, Georg},
    title = {Sufficient conditions for linear stability of complex-balanced equilibria in generalized mass-action systems},
    year = {2024},
    volume = {23},
    number = {1},
    journal = {SIAM Journal on Applied Dynamical Systems},
    doi = {10.1137/22M154260X},
}

@Article{BorosMuellerRegensburger2020,
  author = {Boros, Bal{\'a}zs and M{\"u}ller, Stefan and Regensburger, Georg},
  title  = {Complex-balanced equilibria of generalized mass-action systems: Necessary conditions for linear stability},
  year   = {2020},
  journal = {Mathematical Biosciences and Engineering},
  publisher = {American Institute of Mathematical Sciences ({AIMS})},
  volume = {17},
  number = {1},
  pages = {442--459},
  doi = {10.3934/mbe.2020024}
}

@Article{CraciunMuellerPanteaYu2019,
  author = {{Craciun}, Gheorghe and {M{\"u}ller}, Stefan and {Pantea}, Casian and {Yu}, Polly},
  title  = {{A generalization of Birch's theorem and vertex-balanced steady states for generalized mass-action systems}},
  journal = {Mathematical Biosciences and Engineering},
  Volume  = {16},
  number  = {6},
  Pages    = {8243--8267},
  year   = {2019},
  doi = {10.3934/mbe.2019417}
}

@Article{MuellerHofbauerRegensburger2019,
  author  = {M{\"u}ller, Stefan and Hofbauer, Josef and Regensburger, Georg},
  title   = {On the bijectivity of families of exponential/generalized polynomial maps},
  journal = {SIAM J. Appl. Algebra Geom.},
  year    = {2019},
  Volume = {3},
  number = {3},
  Pages   = {412--438},
  doi = {10.1137/18M1178153}
}

@Article{Mueller2016,
  Title       = {Sign conditions for injectivity of generalized polynomial maps with applications to chemical reaction networks and real algebraic geometry},
  Author   = {M{\"u}ller, Stefan and Feliu, Elisenda and Regensburger, Georg and Conradi, Carsten and Shiu, Anne and Dickenstein, Alicia},
  Journal  = {Found. Comput. Math.},
  Fjournal = {Foundations of Computational Mathematics. The Journal of the Society for the Foundations of Computational Mathematics},
  Volume = {16},
  number = {1},
  Pages   = {69--97},
  Year      = {2016},
  doi = {10.1007/s10208-014-9239-3}
}

@InProceedings{MuellerRegensburger2014,
  Author                   = {M{\"u}ller, Stefan and Regensburger, Georg},
  Title                    = {Generalized Mass-Action Systems and Positive Solutions of Polynomial Equations with Real and Symbolic Exponents},
  Booktitle                = {Computer Algebra in Scientific Computing -- CASC 2014},
  Year                     = {2014},
  Pages                    = {302--323},
  Series                   = {Lecture Notes in Comput. Sci.},
  Volume                   = {8660},
  Publisher                = {Springer},
  Address                  = {Berlin/Heidelberg},
  doi = {10.1007/978-3-319-10515-4_22}
}

@Article{MuellerRegensburger2012,
  Author                   = {M{\"u}ller, Stefan and Regensburger, Georg},
  Title                    = {Generalized mass action systems: {C}omplex balancing equilibria and sign vectors of the stoichiometric and kinetic-order subspaces},
  Journal                  = {SIAM J. Appl. Math.},
  Year                     = {2012},
  Volume                   = {72},
  number                    = {6},
  Pages                    = {1926--1947},
  doi = {10.1137/110847056}
}

@Article{MuellerRegensburger2016,
  Author                   = {M{\"u}ller, Stefan and Regensburger, Georg},
  Title                    = {Elementary vectors and conformal sums in polyhedral geometry and their relevance for metabolic pathway analysis},
  Journal                  = {Front. Genet.},
  Year                       = {2016},
  Volume                  = {7},
  Number                 = {90},
  Pages                    = {1--11},
  doi = {10.3389/fgene.2016.00090}
}

@article{Hernandez2021,
  doi = {10.1007/s11538-021-00906-3},
  url = {https://doi.org/10.1007/s11538-021-00906-3},
  year = {2021},
  publisher = {Springer Science and Business Media {LLC}},
  volume = {83},
  number = {7},
  author = {Bryan S. Hernandez and Ralph John L. De la Cruz},
  title = {Independent Decompositions of Chemical Reaction Networks},
  journal = {Bull. Math. Biol.}
}

@article{Hernandez2023,
  doi = {10.1371/journal.pcbi.1011039},
  url = {https://doi.org/10.1371/journal.pcbi.1011039},
  year = {2023},
  month = apr,
  publisher = {Public Library of Science ({PLoS})},
  volume = {19},
  number = {4},
  pages = {e1011039},
  author = {Bryan S. Hernandez and Patrick Vincent N. Lubenia and Matthew D. Johnston and Jae Kyoung Kim},
  editor = {Mark Alber},
  title = {A framework for deriving analytic steady states of biochemical reaction networks},
  journal = {{PLOS} Computational Biology}
}

@incollection{adleman2014mathematics,
  title={On the mathematics of the law of mass action},
  author={Adleman, L. and Gopalkrishnan, M. and Huang, M. and Moisset, P. and Reishus, D.},
  booktitle={A Systems Theoretic Approach to Systems and Synthetic Biology {I}: Models and System Characterizations},
  pages={3--46},
  year={2014},
  publisher={Springer}
}

@book{Feinberg2019,
  title = {Foundations of Chemical Reaction Network Theory},
  ISBN = {9783030038588},
  ISSN = {2196-968X},
  url = {http://dx.doi.org/10.1007/978-3-030-03858-8},
  DOI = {10.1007/978-3-030-03858-8},
  journal = {Applied Mathematical Sciences},
  publisher = {Springer International Publishing},
  author = {Feinberg,  Martin},
  year = {2019}
}

@article{birch1963maximum,
  title={Maximum likelihood in three-way contingency tables},
  author={Birch, M.},
  journal={J. R. Stat. Soc.},
  volume={25},
  number={1},
  pages={220--233},
  year={1963},
  publisher={Oxford University Press}
}

@article{voit2015150,
  title={150 years of the mass action law},
  author={Voit, E. and Martens, H. and Omholt, S.},
  journal={PLOS Comput. Biol.},
  volume={11},
  number={1},
  pages={e1004012},
  year={2015},
  publisher={Public Library of Science}
}

@article{yu2018mathematical,
  title={Mathematical {A}nalysis of {C}hemical {R}eaction {S}ystems},
  author={Yu, P. and Craciun, G.},
  journal={Isr. J. Chem.},
  volume={58},
  number={6-7},
  pages={733--741},
  year={2018},
  publisher={Wiley Online Library}
}

@article{craciun2015toric,
  title={Toric differential inclusions and a proof of the global attractor conjecture},
  author={Craciun, G.},
  journal={arXiv}, 
  year={2015},
  note = {\href{https://arxiv.org/abs/1501.02860}{arXiv:1501.02860}},
}

@article{craciun2019polynomial,
  title={Polynomial dynamical systems, reaction networks, and toric differential inclusions},
  author={Craciun, G.},
  journal={SIAM J. Appl. Algebra Geom.},
  volume={3},
  number={1},
  pages={87--106},
  year={2019},
  publisher={SIAM}
}

@article{craciun2020endotactic,
  title={Endotactic networks and toric differential inclusions},
  author={Craciun, G. and Deshpande, A.},
  journal={SIAM J. Appl. Dyn. Syst.},
  volume={19},
  number={3},
  pages={1798--1822},
  year={2020},
  publisher={SIAM}
}

@inproceedings{boros2010notes_1,
  title={Notes on the Deficiency One Theorem: single linkage class},
  author={Boros, B.},
  booktitle={Proceedings of the 19th International Symposium on Mathematical Theory of Networks and Systems--MTNS},
  volume={5},
  number={9},
  year={2010}
}

@article{boros2012notes_2,
  title={Notes on the deficiency-one theorem: multiple linkage classes},
  author={Boros, B.},
  journal={Math. Biosci.},
  volume={235},
  number={1},
  pages={110--122},
  year={2012},
  publisher={Elsevier}
}

@article{boros2013dependence,
  title={On the dependence of the existence of the positive steady states on the rate coefficients for deficiency-one mass action systems: single linkage class},
  author={Boros, B.},
  journal={J. Math. Chem.},
  volume={51},
  pages={2455--2490},
  year={2013},
  publisher={Springer}
}

@article{boros2013existence,
  title={On the existence of the positive steady states of weakly reversible deficiency-one mass action systems},
  author={Boros, B.},
  journal={Math. Biosci.},
  volume={245},
  number={2},
  pages={157--170},
  year={2013},
  publisher={Elsevier}
}

@Article{Feinberg1995a,
  Title                    = {The existence and uniqueness of steady states for a class of chemical reaction networks},
  Author                   = {Feinberg, M.},
  Journal                  = {Arch. Rational Mech. Anal.},
  Year                     = {1995},
  Number                   = {4},
  Pages                    = {311--370},
  Volume                   = {132},

  Coden                    = {AVRMAW},
  Doi                      = {10.1007/BF00375614},
  Fjournal                 = {Archive for Rational Mechanics and Analysis},
  ISSN                     = {0003-9527},
  Mrclass                  = {92E20 (34A12 35Q80 80A30)},
  Mrnumber                 = {1365832},
  Mrreviewer               = {Yuan Hua Tang},
  Url                      = {http://dx.doi.org/10.1007/BF00375614}
}

@Article{Feinberg1995b,
  Title                    = {Multiple steady states for chemical reaction networks of deficiency one},
  Author                   = {Feinberg, M.},
  Journal                  = {Arch. Rational Mech. Anal.},
  Year                     = {1995},
  Number                   = {4},
  Pages                    = {371--406},
  Volume                   = {132},

  Coden                    = {AVRMAW},
  Doi                      = {10.1007/BF00375615},
  Fjournal                 = {Archive for Rational Mechanics and Analysis},
  ISSN                     = {0003-9527},
  Mrclass                  = {92E20 (35Q80 80A30)},
  Mrnumber                 = {1365833},
  Url                      = {http://dx.doi.org/10.1007/BF00375615}
}

@Article{Feinberg1987,
  Title                    = {{Chemical reaction network structure and the stability of complex isothermal reactors--{I}. The deficiency zero and deficiency one theorems}},
  Author                   = {Feinberg, M.},
  Journal                  = {Chem. Eng. Sci.},
  Year                     = {1987},
  Number                   = {10},
  Pages                    = {2229--2268},
  Volume                   = {42},
  doi = {10.1016/0009-2509(87)80099-4}
}

@Article{Feinberg1972,
  Title                    = {Complex balancing in general kinetic systems},
  Author                   = {Feinberg, M.},
  Journal                  = {Arch. Rational Mech. Anal.},
  Year                     = {1972},
  Pages                    = {187--194},
  Volume                   = {49},

  Fjournal                 = {Archive for Rational Mechanics and Analysis},
  ISSN                     = {0003-9527},
  Mrclass                  = {82.35},
  Mrnumber                 = {0413930},
  Mrreviewer               = {Ole J. Heilmann},
  doi = {10.1007/BF00255665}
}

@article{gunawardena2003chemical,
  title={Chemical reaction network theory for in-silico biologists},
  author={Gunawardena, J.},
  journal={Notes available for download at 
  https://vcp.upf.edu/papers.html\#preprints},
  year={2003}
}

@Article{Horn1972,
  Title                    = {Necessary and sufficient conditions for complex balancing in chemical kinetics},
  Author                   = {Horn, F.},
  Journal                  = {Arch. Rational Mech. Anal.},
  Year                     = {1972},
  Pages                    = {172--186},
  Volume                   = {49},

  Fjournal                 = {Archive for Rational Mechanics and Analysis},
  ISSN                     = {0003-9527},
  Mrclass                  = {82.35},
  Mrnumber                 = {0413929},
  Mrreviewer               = {Ole J. Heilmann},
  doi = {10.1007/BF00255664}
}

@Article{HornJackson1972,
  Title                    = {General mass action kinetics},
  Author                   = {Horn, F. and Jackson, R.},
  Journal                  = {Arch. Rational Mech. Anal.},
  Year                     = {1972},
  Pages                    = {81--116},
  Volume                   = {47},

  Fjournal                 = {Archive for Rational Mechanics and Analysis},
  ISSN                     = {0003-9527},
  Mrclass                  = {80.34},
  Mrnumber                 = {0400923},
  Mrreviewer               = {J. Adler},
  doi = {10.1007/BF00251225}
}

% ========= ========= ========= ========= ========= ========= ========= =========

\clearpage

\begin{appendix}

\section*{Appendix}

% ========= ========= ========= ========= ========= ====

\section{Proof of Theorem~\ref{thm:existence}} \label{app:proof_existence}

\begin{proof}
In the proof of Theorem~\ref{thm:one_class}, we have shown that $\bar b_1\cdot \bar b_\omega <0$ implies $|Y_c| \ge 1$ for all~$c$. 
We now show the other direction, $|Y_c| \ge 1$ for all~$c$ implies $\bar b_1\cdot \bar b_\omega <0$. 
%Let us assume for contradiction that $\bar b_1\cdot \bar b_\omega \ge 0$. We will then show that there exists a $c^*$ such that $Y_{c^*} = \emptyset$.

Recall that \blue{determining $Y_c$} is equivalent to solving
$f(t) = c^*$ \blue{for $t \in (-1,1)$,} 
where 
\[
f(t) = \prod_{i=1}^\blue{\omega} (1+t\bar q_i)^{\bar b_i} 
\]
\blue{and $1 = \bar q_1  > \cdots > \bar q_\omega=-1$.
More explicitly, $f(t) = (1+t)^{\bar b_1} \cdot \, \cdots \, \cdot (1-t)^{\bar b_\omega} > 0$.

Hence, $|Y_c|\ge1$ for all $c$ is equivalent to $|\{ t \in (-1,1) \mid f(t)=c^* \}| \ge 1$ for all $c^*>0$.
Now, for all $\delta>0$, the function $f$ attains positive minimum and maximum values on the compact interval $[-1+\delta, 1 - \delta]$, and $f$ being surjective onto $(0,\infty)$ requires $\displaystyle\lim_{t \to -1} f(t) = 0$ and $\displaystyle\lim_{t \to 1} f(t) = \infty$ (or vice versa). This is equivalent to $\bar b_1 > 0$ and $\bar b_\omega < 0$ (or vice versa), that is $\bar b_1\cdot \bar b_\omega <0$.}
\end{proof}

% ========= ========= ========= ========= ========= ========= ========= =========

\section{Reaction networks}

\subsection{Index notation} \label{app:index}

We explicitly state the incidence and source matrices
as well as the \blue{(rectangular)} Laplacian matrix \blue{of a labeled simple digraph $(V,E,k)$, using} index notation.

\begin{enumerate}[(i)]
\item 
The incidence matrix
$I_E \in \R_>^{V \times E}$ 
is given by
\[
(I_E)_{i,j \to j'} =
\begin{cases}
-1 , & \text{if } i = j , \\
1 , & \text{if } i = j' , \\
0 , & \text{otherwise.}
\end{cases}
\]
\item 
The source matrix
$I_{E,s} \in \R_>^{V_s \times E}$ is given by
\[
(I_{E,s})_{i,j \to j'} =
\begin{cases}
1 , & \text{if } i = j , \\
0 , & \text{otherwise.}
\end{cases}
\]
\item 
The rectangular Laplacian matrix $R_k \in \R^{V \times V_s}$ is given by
\[
(R_k)_{i,j} =
\begin{cases}
k_{j \to i} , & \text{if } (j \to i) \in E , \\
-\sum_{(i \to i') \in E} k_{i \to i'} , & \text{if } i = j, \\
0 , & \text{otherwise.}
\end{cases}
\]
It can be decomposed as $R_k = I_E \diag(k) (I_{E,s})^\trans$. 
\end{enumerate}

% \subsection{Auxiliary graph} \label{app:aux}

% We recall the definition of an auxiliary graph from~\cite{Mueller2022}. 

% \begin{definition}
% \blue{
% For a connected simple digraph $G=(V,E)$, 
% we call a graph $G_{\mathcal{E}} = (V,\mathcal{E})$ {\em auxiliary},
% if $|\mathcal{E}| = |V|-1$ for $\mathcal{E}\subseteq V\times V$ and $G_{\mathcal{E}}$ is connected.
% }
% The incidence matrix \blue{of} $G_{\mathcal{E}}$ is given \blue{by}
% \[
% I_{\mathcal{E}}(i,j \to j') =
% \begin{cases}
% -1 , & \text{if } i = j , \\
% 1 , & \text{if } i = j' , \\
% 0 , & \text{otherwise.}
% \end{cases}
% \]    
% \end{definition}

% ========= ========= ========= ========= ========= ========= ========= =========

\subsection{Proof of the second salt theorem} \label{app:salt}

We prove the second salt theorem for rectangular (rather than square) Laplacian matrices.

\begin{proof}[Proof of Lemma~\ref{lem:salt}]
Let $T' = \{ 1,\ldots, t' \} \subseteq T$.  
The equations $R_k \, \hat q = 0$ and $R_k \, \ones_{V_s} = \beta$ imply
\begin{align*}
\sum_{\substack{(i \to i') \in \\ V \setminus T'\to T'}} k_{i\to i'} \, \hat q_i - \sum_{\substack{(i' \to i) \in \\ T'\to V \setminus T'}} k_{i'\to i} \, \hat q_{i'} 
&= 0 , \\
\sum_{\substack{(i \to i') \in \\ V \setminus T'\to T'}} k_{i\to i'} 
- \sum_{\substack{(i' \to i) \in \\ T'\to V \setminus T'}} k_{i'\to i} 
&= \sum_{i\in T'} \beta_i .
\end{align*}
If $T' \subset T$, then the first equation and the order on $\hat q$ imply
\[
\hat q_{t'+1} \sum_{\substack{(i \to i') \in \\ V \setminus T'\to T'}} k_{i\to i'} 
\ge \sum_{\substack{(i \to i') \in \\ V \setminus T'\to T'}} k_{i\to i'} \, \hat q_i 
= \sum_{\substack{(i' \to i) \in \\ T'\to V \setminus T'}} k_{i'\to i} \, \hat q_{i'} 
\ge \hat q_{t'} \sum_{\substack{(i' \to i) \in \\ T'\to V \setminus T'}} k_{i'\to i} ,
\]
and the second equation and $\hat q_{t'} \ge \hat q_{t'+1}$ further imply
\[
\sum_{i\in T'} \beta_i = \sum_{\substack{(i \to i') \in \\ V \setminus T'\to T'}} k_{i\to i'} - \sum_{\substack{(i' \to i) \in \\ T'\to V \setminus T'}} k_{i'\to i} 
\ge \left(\frac{q_{t'}}{q_{t'+1}} - 1\right) \sum_{\substack{(i' \to i) \in \\ T'\to V \setminus T'}} k_{i'\to i}
\ge 0 .
\]
In particular, $\sum_{i\in T'} \beta_i > 0$ if $\hat q_{t'} > \hat q_{t'+1}$.

If $T' = T$, then $(T \to V \setminus T) = \emptyset$, since $T$ is an \blue{absorbing} strong component, and 
\[
\sum_{i\in T} \beta_i = \sum_{\substack{(i \to i') \in \\ V \setminus T\to T}} k_{i\to i'} \ge 0 .
\]
The sum is zero if and only if $(V \setminus T\to T) = \emptyset$,
that is, 
$T=V$.
\end{proof}

% ========= ========= ========= ========= ========= ========= ========= =========

\end{appendix}

\end{document}